\documentclass[12pt]{article}
\usepackage[english]{babel}
\usepackage{amssymb,amsmath}
\parskip\medskipamount
\newtheorem{theorem}{Theorem}[section]

\newtheorem{prop}[theorem]{Proposition}
\newtheorem{defi}[theorem]{Definition}
\newtheorem{lemma}[theorem]{Lemma}
\newtheorem{cor}[theorem]{Corollary}

\newtheorem{rem}[theorem]{Remark}

\newenvironment{proof}{\par\noindent\textbf{Proof}\hspace{1em}}{\qed\bigskip}
\def\<{\langle}
\def\>{\rangle}
\newcommand{\cC}{\mathcal{C}}

\newcommand{\PG}{\mathsf{PG}}
\newcommand{\GL}{\mathsf{GL}}

\newcommand{\K}{\mathbb{K}}

\newcommand{\cS}{\mathcal{S}}

\newcommand{\cW}{\mathcal{W}}

\newcommand{\cG}{\mathcal{G}}

\newcommand{\cL}{\mathcal{L}}

\newcommand{\ad}{\mathrm{adj}}
\def\qed{{\hfill\hphantom{.}\nobreak\hfill$\Box$}}

\begin{document}

\author{J.~Schillewaert \and H.~Van Maldeghem}
\title{Projective planes over quadratic 2-dimensional algebras}
\date{}
\maketitle


\begin{abstract}
The split version of the Freudenthal-Tits magic square stems from Lie theory and constructs a Lie algebra starting from two split composition algebras \cite{Freudenthal,Tits1,Tits2}. The geometries appearing in the second row are Severi-Brauer varieties \cite{Zak}.  We provide an easy uniform axiomatization of these geometries and related ones, over an arbitrary field. In particular we investigate the entry $A_{2}\times A_{2}$ in the magic square, characterizing Hermitian Veronese varieties, Segre varieties and embeddings of Hjelmslev planes of level 2 over the dual numbers.  In fact this amounts to a 
common characterization of ``projective planes over quadratic 2-dimensional algebras'', in casu the 
split and non-split Galois extensions and the dual numbers. 
\end{abstract}

\section{Introduction, Notation and Main Results}
\label{sec:Introduction}
\subsection{Mazzocca-Melone axioms and $\cC$-Veronesean sets}

In this paper we present a far-going generalization of the Mazzocca-Melone approach to quadric Veronese varieties. In this introduction, we describe the formal situation, and mention some history. The compelling motivation for our approach (why exactly generalizing in the way we do) is explained in the final section of the paper in order not to interfere with the mathematical flow of the paper. The reader might want to read the final section first. It puts our result in the broader perspective of the Freudenthal-Tits Magic Square \cite{Freudenthal,Tits1,Tits2}, certain alternative algebras, and representations of a class of spherical buildings (contaning those having exceptional type $\mathsf{E}_6$) in projective space. Let us just mention here that our intension is not just ``generalizing'', but rather a new geometric approach to the aforementioned magic square. 

The Mazzocca-Melone axioms for Veronese varieties have proved to be of fundamental 
importance for the theory of Veronese varieties. Mazzocca and Melone \cite{Maz-Mel:84} were 
the first ones to prove such a characterization (for quadric Veronese varieties of finite projective 
planes), and the same axioms, with only some minor changes depending on the context, were 
used by others to characterize finite quadric Veronese varieties of projective spaces \cite{Tha-Mal:04b}, 
quadric Veronese varieties in general \cite{Sch-Mal:**}, finite Hermitian Veronese varieties 
\cite{Tha-Mal:05}, and Hermitian Veronese varieties in general \cite{Sch-Mal:11}. In this paper, we introduce a further 
minor change in these axioms to include more varieties over arbitrary fields. 

Let us first introduce the (flavor of the) Mazzocca-Melone axioms. Let $n\geq 1$ be a natural 
number. Let us call a point set $S$ of $\PG(n,\K)$, $\K$ any skew field, a \emph{hypersurface} if for 
any point $x\in S$ the union $T_x(S)$ of the set of lines through $x$ that either are contained in 
$S$ or intersect $S$ precisely in $\{x\}$ either forms a hyperplane of $\PG(n,\K)$ (and then $x$ is 
called a \emph{regular point} of $S$), or is the whole point set of $\PG(n,\K)$ (and then $x$ is 
called a \emph{singular point} of $S$).   Let $\mathcal{C}$ be a class of hypersurfaces of 
$\PG(n,\K)$. Let there be given a set $X$ of points spanning some projective space 
$\mathrm{PG}(N,\mathbb{K})$, with $N>n$ and $\mathbb{K}$ still any skew field, and let $\Xi$ be a collection 
of subspaces of $\mathrm{PG}(N,\mathbb{K})$ of the same dimension $n$, such that, for any 
$\xi\in\Xi$, the intersection $\xi\cap X$ is the set $X(\xi)$ of regular points of some hypersurface 
$\overline{X(\xi)}$ of $\xi$ belonging to the class $\mathcal{C}$ (and then, for $x\in X(\xi)$, we sometimes denote 
$T_x(X(\xi))$ simply by $T_x(\xi)$). Put $\overline{X}$ equal to the union of all $\overline{X(\xi)}$, $\xi\in\Xi$. 
We call $X$ a \emph{$\mathcal{C}$-Veronesean set} if the following properties hold :
\begin{itemize}
\item[(V1)] Any two points $x,y\in X$ lie in at least one element of $\Xi$, denoted by $[x,y]$, if unique.

\item[(V2)] If $\xi_1,\xi_2\in \Xi$, with $\xi_1\neq \xi_2$, then $\xi_1\cap\xi_2\subset \overline{X(\xi_1)}\cap\overline{X(\xi_2)}$, and $\xi_1\cap\xi_2\cap (\overline{X}\setminus X)$ is contained in some subspaces of $\xi_1\cap\xi_2$ of codimension $1$.

\item[(V3)] If $x\in X$ and $\xi\in\Xi$, with $x\notin \xi$ and such that for each $y\in\xi\cap X$, there is a unique element of $\Xi$ containing both $x$ and $y$, then each of the subspaces $T_x([x,y])$, 
$y\in \xi\cap X$, is contained in a fixed $(2n-2)$-dimensional subspace of 
$\mathrm{PG}(N,\mathbb{K})$, denoted by $T(x,\xi)$.
\end{itemize}
  
The classification of $\cC$-Veronesean sets is usually achieved by first defining an index (which is 
a natural number), then considering the case of index 2, and finally invoking an induction argument 
on the index. For the case of index 2, one can consider the following alternative axiom to (V3).

\begin{itemize}
\item[(V3*)] If $x\in X$, then each of the subspaces $T_x(\xi)$, with $x\in\xi\in \Xi$, is contained in a 
fixed $(2n-2)$-dimensional subspace of $\mathrm{PG}(N,\mathbb{K})$, denoted by $T(x,\pi)$.
\end{itemize}

The original Mazzocca-Melone axioms amount to the case $n=2$ and $\cC$ the class of 
irreducible finite plane conics. Mazzocca \& Melone characterized finite quadric Veronese varieties over fields of 
odd order \cite{Maz-Mel:84}. This was later generalized by Hirschfeld \& Thas to include all finite fields \cite{Hir-Tha:91}. Thas \& Van Maldeghem \cite{Tha-Mal:04} then considered the case where $\cC$ is the class of finite 
plane ovals (and proved each of these ovals must automatically be a conic), whereas 
Schillewaert \& Van Maldeghem \cite{Sch-Mal:**} classified $\cC$-Veronesean sets for $\cC$ the class of all ovals of 
any (finite or infinite) projective plane. 

Cooperstein, Thas \& Van Maldeghem \cite{Coo-Tha-Mal:04} characterized finite Hermitian Veronese 
varieties as the $\cC$-Veronesean sets, where $\cC$ is the class of elliptic quadrics of finite $3$-space. Thas \& Van Maldeghem 
\cite{Tha-Mal:05} generalized this to the class of ovoids of finite $3$-space, whereas Schillewaert \& Van Maldeghem 
\cite{Sch-Mal:11} proved the same characterization 
for Hermitian Veronese varieties over any field, as $\cC$-Veronesean sets, where $\cC$ is the class of all 
ovoids of $\PG(3,\K)$, with $\K$ any skew field. 

In the present paper, we want to consider $\cC$-Veronesean sets for $\cC$ either the class of ruled 
quadrics of $\PG(3,\K)$, with $\mathbb{K}$ an arbitrary field, or the class of oval cones in $\PG(3,\K)$, with 
$\K$ any skew field. In the former case, we also speak of a \emph{Segrean Veronesean set}, 
whereas in the latter case we speak of a \emph{Hjelmslevian Veronesean set}. 

Let $X$ be a Hermitian, Segrean or Hjelmslevian Veronesean set. We define the geometry $\cG(X)$ 
as follows. The points are the elements of $X$, the lines are the members of $\Xi$ and incidence is 
given by containment.  Note that, for $X$ a Hermitian Veronesean set, the geometry $\cG(X)$ is a 
projective plane over a quadratic Galois extension of $\K$, see \cite{Sch-Mal:11}. But for the Segrean and Hjelmslevian cases, 
$\cG(X)$ will be a ring geometry. The corresponding ring will be a $2$-dimensional algebra over a field. So let us have a look at such algebras in the next subsection. 

\subsection{Planes over some 2-dimensional algebras}\label{planes}

Let $\K$ be a (commutative) field and let $V$ be a $2$-dimensional commutative and associative 
algebra over $\K$ such that each vector $v$ of $V$ is the product $v=a\times b$ of two vectors 
$a,b$ of $V$. Let $V_0$ be the set of zero-divisors of $V$ and put $V^*=V\setminus V_0$. 

We define the following geometry $\cG(V)$. The points are the classes of triples $V^*(x,y,z)$, 
$x,y,z\in V$, such that, if $v\in V$ and $v\times(x,y,z)=(0,0,0)$, then $v=0$. The lines are the 
classes of triples $V^*[a,b,c]$, $a,b,c\in V$, such that, if $v\in V$ and $v\times[a,b,c]=[0,0,0]$, then 
$v=0$. A point $V^*(x,y,z)$ is incident with a line $V^*[a,b,c]$ if 
$a\times x + b\times y + c\times z=0$. We will usually omit the multiplication sign ``$\times$'' in the sequel.

A $3\times3$ matrix with entries in $V$ will be said to \emph{have rank $1$} if for every two rows $R_1,R_2$, 
there are elements $a,b\in V$ such that $aR_1+bR_2$ is the zero-row, and such that no nonzero element $c$ of $V$
exists such that $ac=bc=0$.  Below we will only use it for ``Hermitian matrices'', where it is clear that one can substitute ``row'' by ``column'' in this definition without changing the meaning.  

There are always three ways to represent $\cG(V)$ in an $8$-dimensional projective space over $\K$, and sometimes there is a fourth. 

First we note that $V$ contains a unique identity element $\mathbf{1}$ with 
respect to $\times$ and there exists a unique (linear) automorphism $\sigma$ of $V$ of order at 
most 2 fixing $\mathbf{1}$ and such that both $v+v^\sigma$ and $v\times v^\sigma$ belong to 
$\K\cdot\mathbf{1}$, for all $v\in V$ (see Section~\ref{algebras}). We will henceforth identify $\K\cdot\mathbf{1}$ with $\K$. 

\begin{enumerate}\item Let $V$ and $\sigma$ be as above. The set of points of $\cG(V)$ is in 
bijective correspondence with the set of rank 1 Hermitian $3\times 3$ matrices over $V$ (i.e., 
diagonal elements belong to $\K\cdot\mathbf{1}$ and corresponding symmetric elements are 
images under $\sigma$), up to a scalar multiple, by mapping a point $V^*(x,y,z)$ onto the matrix 
$(x~y~z)^t(x~y~z)^\sigma$, where $t$ denotes transposition. Viewing the set of all Hermitian 
$3\times 3$ matrices over $V$ as a $9$-dimensional vector space over $\K$, we see that this 
defines a point set $X$ in $\PG(8,\K)$. It is easy to see that the lines of the geometry correspond to 
images of rank 1 Hermitian $2\times 2$ matrices over $V$, and these in turn define quadrics in 
$\PG(3,\K)$. This hints to the fact that $X$ is indeed a $\cC$-Veronesean set, where $\cC$ is either 
the class of all hyperbolic quadrics, or the class of all quadratic cones, or an isomorphism class of 
elliptic quadrics. For ease of reference, we call the set $X$ here defined a \emph{$V$-set defined 
by matrices}.

\item
Consider a 3-dimensonal module $\cW$ over $V$, and reconsider $\cW$ as a 6-dimensional vector 
space over $\K$. The set of points of $\cG(V)$ defines a family of one-dimensional subspaces over 
$V$ of $\cW$, and this corresponds to a set of 2-dimensional subspaces over $\K$. This yields a 
set $\cL$ of lines in the projective 5-space over $\K$. If we take the line Grassmannian of $\cL$, 
then we end up in $\PG(14,\K)$. But actually, one can prove (this follows from Proposition~\ref{span8} below, combined with Propositions~\ref{theoV1} and~\ref{theoV2}) that the 
image $X$ of $\cL$ under the line Grassmannian spans a subspace $\PG(8,\K)$ of dimension $8$.  
We will refer to $X$ as a \emph{$V$-set defined by reduction}.

To see the lines, one must consider certain $2$-dimensional submodules of $\cW$ over $V$, 
which correspond to certain $4$-dimensional subspaces of $\cW$ over $\K$. Projectively, we have 
a line set in a projective $3$-space, and the Grassmannian then lies on the Klein quadric; in fact 
we obtain a quadric in projective $3$-space (again either an elliptic quadric, a ruled quadric, or a 
quadratic cone).   

\item
There is another construction leading to the same set of lines of the previous paragraph before 
taking Grassmannian. Indeed, we will show in Section~\ref{algebras} below that the algebra $V$ is a matrix algebra, a 
subalgebra of the full $2\times 2$ algebra over $\K$. Now the set of lines in the projective 5-space 
obtained above can also be obtained by juxtaposition of three arbitrary $2\times 2$ matrices 
(corresponding to triples satisfying the same restrictions as in the definition of points of $\cG(V)$) of 
the corresponding matrix algebra, and then taking the joins of the points represented by the rows of 
the $6\times 2$ matrices thus obtained. The set $X$ obtained by taking the image under the line 
Grassmannian will be referred to as a \emph{$V$-set defined by juxtaposition}. 

\item We provide a parameter representation of $\cG(X)$. Let $\zeta$ denote a fixed invertible element of $V$ not fixed under $\sigma$; if 
this cannot be found, then any invertible element not in $\K\times 1$ will do. Then we let the point $V^*(x,y,z)$ correspond to the point in 
$\PG(8,\K)$ having coordinates $$(x^\sigma x, y^\sigma y, z^\sigma z, x^\sigma y+y^\sigma x, y^\sigma z+z^\sigma y, z^\sigma x+x^\sigma z,
 \zeta x^\sigma y+\zeta^\sigma y^\sigma x, \zeta y^\sigma z+\zeta^\sigma z^\sigma y, \zeta z^\sigma x+\zeta^\sigma x^\sigma z),$$ and one 
checks easily that this is independent of the chosen representative. We call this correspondence the \emph{Veronese correspondence}. Then, by 
definition, the set of images of points of $\cG(V)$ under the Veronese correspondence is $X$. The lines of $\mathcal{G}(X)$ are the 
images under the Veronese correspondence of the lines of $\cG(V)$. Since each line involves two free parameters over $V$, hence four over 
$\K$, every line will be contained in a projective $3$-subspace, and hence here we also see that this hints to a $\mathcal{C}$-Veronesean 
set with $\mathcal{C}$ the a set of quadrics in projective $3$-space.  We call the set $X$ thus constructed a \emph{$V$-set defined by 
parametrization}. 
\end{enumerate}  

\subsection{Main Result}

Our main result connects the abstractly defined ring geometries with the notion of $\cC$-Veronesean set as follows.

\textbf{Main Result.} \emph{Let $\K$ be any field. Let $V$ be a  $2$-dimensional commutative and 
associative algebra over $\K$ such that each vector of $V$ is the product of two vectors of $V$. 
Then the $V$-set defined by matrices is isomorphic to the $V$-set defined by reduction, and also to 
the $V$-set defined by juxtaposition (which allows to briefly talk about $V$-sets). These $V$-sets 
are isomorphic to the $V$-set defined by parametrization as soon as $V$ does not contain an element 
whose square is $0$, and $V$ is not an inseparable quadratic extension of $\K$.  Also, such a $V$-set 
is either a Hermitian Veronesean, or a Segrean 
Veronesean, or a Hjelmslevian Veronesean set $X$, and $\cG(V)$ is isomorphic to $\cG(X)$. 
Conversely, every Hermitian, Segrean or Hjelmslevian Veronesean set in $\PG(8,\K)$ can be obtained 
from a $V$-set over $\K$ and hence is unique, up to projectivity. Also, a Segrean Hermitian set is 
projectively equivalent to a Segre variety 
of type $(2,2)$ and the geometry $\cG(X)$ for $X$ a Hjelmslevian Veronesean set is a projective Hjelmslev plane of level $2$ over the dual numbers over $\K$.  }

For the definitions of Segre variety and projective Hjelmslev plane of level 2, we refer to the beginning of Section~\ref{Segre}, and Proposition~\ref{propH2}, respectively. 

We are going to prove the Main Result in small pieces. The main work is done in the next two sections, where we first classify Segrean Hermitian sets, and then prove that  Hjelmslevian Veronesean sets are projectively unique. After that, we take a look at the different possibilities for $V$, given $\K$, and prove the equivalence of the various $V$-sets. We then pause and prove some properties of the ring geometries $\cG(V)$. At last, we show that each $V$-set  is an appropriate $\mathcal{C}$-Veronesean set, establishing the Main Result. 

In the final section we put our investigations in a broader perspective, thus motivating the results of the present paper.


\section{Segrean Veronesean Sets} \label{Segre}
In this section, our main goal is to characterize \emph{Segre varieties of type $(2,2)$} (we will prove a slightly more general result
which also includes Segre varieties of type $(1,2)$ and of type $(1,3)$). 
Such a variety is the image $\cS_{2,2}$ of the direct product $\PG(2,\K)\times\PG(2,\K)$ 
of the point sets of two isomorphic projective planes over a commutative field $\K$ under the mapping

$$\begin{array}{ll}\sigma: & \PG(2,\K)\times\PG(2,\K)\rightarrow \PG(8,\K):\\ & ((x,y,z),(x',y',z'))\mapsto (xx',xy',xz',yx',yy',yz',zx',zy',zz').\end{array}$$

(In general, for the definition of a Segre variety $\cS_{n,m}$ of type $(m,n)$, $m,n\geq 1$, one considers 
the direct product of the point sets of a projective $m$-space with a projective $n$-space and 
the obvious generalization of the mapping above.)

One observes that the image of the direct product of two lines inside these planes is a 
hyperbolic quadric in some $3$-dimensional subspace of $\PG(8,\K)$. Clearly, every pair of points in the 
image is contained in at least one such hyperbolic quadric. Note that the automorphism group of 
$\cS_{2,2}$ is transitive on the points, and on the hyperbolic quadrics (indeed, if a linear 
collineation acts on the first component with matrix $A$, and one on the second component with 
matrix $B$, then the tensor product $A\otimes B$ acts in a natural way on $\cS_{2,2}$). Thus we 
can choose the coordinates in an easy way to verify that (1) the intersection of two arbitrary 3-spaces each containing distinct such hyperbolic quadrics is a subset of $\cS_{2,2}$, and (2) all 
generators of all hyperbolic quadrics through  a fixed point are contained in a $4$-space. We now 
take these properties as axioms to characterize $\cS_{2,2}$. This gives rise to the Mazzocca-Melone axioms for this case, and we provide some more details now.

A \emph{hypo} $H$ in a $3$-dimensional projective space $\Sigma$ (over $\K$) is the set of points of 
$\Sigma$ on some hyperbolic quadric. For every point $x\in H$, there is a unique plane $\pi$ 
through $x$ intersecting $O$ in two intersecting lines both of which contain $x$. The plane $\pi$ 
contains all lines through $x$ that meet $H$ in only $x$ and is called the \emph{tangent plane} at 
$x$ to $H$ and denoted $T_x(H)$. 

Let $X$ be a spanning point set of $\mathrm{PG}(N,\mathbb{K})$, $N>3$, and let $\Xi$ be a set 
of 3-dimensional projective subspaces of $\mathrm{PG}(N,\mathbb{K})$, called the 
\emph{hyperbolic spaces} of $X$, such that, for any $\xi\in\Xi$, the intersection $\xi\cap X$ 
is a hypo $X(\xi)$ in $\xi$ (and then, for $x\in X(\xi)$, we sometimes denote $T_x(X(\xi))$ 
simply by $T_x(\xi)$). We call $X$ a \emph{Segrean Veronesean set (of index $2$)} if the following properties hold :
\begin{itemize}
\item[(S1)] Any two points $x$ and $y$ lie in at least one element of $\Xi$, which is denoted by $[x,y]$ if it is unique.

\item[(S2)] If $\xi_1,\xi_2\in \Pi$, with $\xi_1\neq \xi_2$, then $\xi_1\cap\xi_2\subset X$.

 


\item[(S3*)] For each $x\in X$, all planes $T_x([x,y])$, $y\in X\setminus\{x\}$, are contained in a fixed 4-dimensional subspace of $\mathrm{PG}(N,\mathbb{K})$, denoted by $T_x$.
 \end{itemize}

Note that (S1) immediately implies that the cardinality of $\Xi$ is at least 2.

The rest of this section is devoted to prove the following theorem.

\begin{theorem} \label{theo1} The point set $X$ of a Segrean Veronesean set of index $2$ is the point set of a Segre variety $\cS_{1,2}$ (and then $N=5$), $\cS_{1,3}$ (and then $N=7$) or $\cS_{2,2}$ (and then $N=8$).
\end{theorem}

From now on, we let $X$ be a set in $\PG(N,\K)$, and $\Xi$ a family of 3-spaces of $\PG(N,\K)$ such that every member of $\Xi$ intersects $X$ in a hypo, and satisfying (S1) and (S2). For the time being, we do not put a restriction on $N$, but we shall always mention our assumption on $N$. Note that we want $X$ to contain at least two hypos, which implies that we can assume $N\geq 5$. Also, we try to prove as much as we can without Axiom (S3*), the reason being that this will help in the general situation when we invoke (S3) instead of (S3*). Hence, if not explicitly mentioned otherwise, we do not assume (S3*). Concerning notation, we will denote the (projective) span of a set $S$ of points of $\PG(N,\K)$ by $\<S\>$. 

Our first aim is to show that $X$ contains a plane. So for the time being, we assume that no plane is entirely contained in $X$. This has a few immediate consequences, the most important of which is the following, where we call a line of $\PG(N,\K)$ entirely contained in $X$ a \emph{singular} line. Similarly for \emph{singular plane}. We phrase the first lemma general enough so that we can still use it when we do have planes.

\begin{lemma} \label{lemma1}\begin{itemize}\item[$(1)$] Every pair of intersecting singular lines not contained in a singular plane is contained in a unique hypo. \item[$(2)$] Every quadrangle of singular lines such that no two consecutive sides is contained in a plane, is contained in a unique hypo, determined by any pair of intersecting lines of that quadrangle.\end{itemize}
\end{lemma}

\begin{proof} Suppose by way of contradiction that the intersecting singular lines $L_1,L_2$ are not contained in a hypo. Let $x$ be any point of the plane $\pi$ spanned by $L_1$ and $L_2$. Then we can choose two lines $M_1,M_2$ through $x$ meeting both $L_1$ and $L_2$ in distinct points. By (S1), the line $M_i$, $i=1,2$, is contained in a hyperbolic space $\xi_i$. If $\xi_1=\xi_2$, then it would contain $L_1\cup L_2$, contradicting our hypothesis. Hence $\xi_1\neq\xi_2$ and so $x\in\xi_1\cap\xi_2$ belongs to $X$ by (S2). Consequently $\pi$ is a singular plane, a contradiction. 

Hence $L_1,L_2$ are contained in a hypo. This hypo is unique as otherwise the plane $\<L_1,L_2\>$ would be singular as it would belong to the intersection of two distinct hyperbolic spaces.  

The second assertion follows from considering the hypos through two opposite pairs of intersecting sides of the quadrangle, and applying the first part of the lemma, noting that opposite points of the quadrangle do not lie on a common singular line. 
\end{proof}

An easy consequence is that there is no triple of coplanar singular lines.

Next we want to show that the singular lines through a point $x$ span a 4-dimensional space (necessarily agreeing with $T_x$), assuming (S3*). We first show that $T_x$ contains at least four lines. We phrase the result independently of (S3*).

\begin{lemma} \label{lemma2}If $X$ does not contain singular planes, and $N\geq 5$, then through every point $x\in X$, there exist at least four singular lines. Consequently $N\geq 6$.
\end{lemma}

\begin{proof}
By considering an arbitrary hypo $H$ through $x$, which exists by (S1) and the fact that $X$ 
generates $\PG(N,\K)$, $N\geq 5$, we see that there are at least two singular lines through $x$. 
We can now choose a point $y_1\in X\setminus\<H\>$ and consider the hypo $H_1$ in $[x,y_1]$. 
Axiom~(S2) implies that $H\cap H_1$ does not contain the two singular lines of $H$ through $x$. 
Hence there is at least one more singular line through $x$. If no more singular lines through $x$ 
existed, Lemma~\ref{lemma1} would imply that there are exactly three hypos containing $x$. 

Now choose a point $y\in H$ not on a singular line with $x$. Interchanging the roles of $x$ and $y$ 
in the previous paragraph, we can choose a singular line $L$ through $y$ not contained in $H$. Let 
$z_1,z_2$ be two distinct points on $L$. If $[x,z_1]=[x,z_2]$, then $[x,z_1]\cap H$ contains the line 
$\<x,y\>$, a contradiction. Hence $[x,z_1]\neq [x,z_2]$. This implies that, if $|\K|\geq 3$, then we 
have at least $4$ hypos through $x$, proving our lemma. Indeed, since in this case we can consider 2 points $z_{1},\:z_{2}$ different from $y$ and not in $H_{1}$. 

Now suppose $|\K|=2$. If $z$ is a point on a singular line with $x$, then, clearly, the number of 
singular lines distinct from $\<x,y\>$  through $z$ is equal to the number of singular lines distinct 
from $\<x,z\>$ through $x$ (indeed, a bijection is given by ``being contained in the same hypo with 
$\<x,z\>$'', and use Lemma~\ref{lemma1}). By connectivity, we obtain that through every point there 
are a constant number of singular lines. If this constant is equal to $3$, then one counts $|X|=19$, 
and a double count of the pairs (point,hypo), where the point is in the hypo, results in 
$19\times 3=n\times 9$, where $n$ is the total number of hypos. Clearly a contradiction. 

So we have at least four singular lines $L_1,L_2,L_3,L_4$ through the point $x\in X$.  Now (S2) 
implies that the hyperbolic spaces containing $L_1,L_2$, and $L_3,L_4$, respectively, meet only in $\{x\}$. 
Consequently $N\geq 6$.

The lemma is proved. 
\end{proof}

The last paragraph of the proof of the previous lemma also proves the following lemma.

\begin{lemma} \label{lemma3}If $X$ does not contain singular planes, and $N\geq 5$, then four 
arbitrary singular lines through a point $x\in X$ span a $4$-space. \qed
\end{lemma}


Now we fix a hypo $H$. We want to study the projection of $X\setminus H$ from $\<H\>$ onto some $(N-4)$-dimensional subspace $F$. In order to do so, we first prove some additional lemmas.

\begin{lemma}\label{lemma4} \label{cor1}Suppose $X$ does not contain singular planes. Let $H$ 
be a hypo, and let $L_1$ and $L_2$ be two distinct singular lines of $X$ meeting $H$ in the points 
$x_1,x_2$, respectively. Then the subspace generated by $H,L_1,L_2$ is $5$-dimensional.
\end{lemma}

\begin{proof} If $x_1=x_2$, then this follows from Lemma~\ref{lemma3}. Now suppose 
$x_1\neq x_2$, and assume that the subspace generated by $H,L_1,L_2$ is $4$-dimensional. 
Then the 
$3$-space $\<L_1,L_2\>$ intersects $\<H\>$ in a plane $\pi$. It is easy to see that $\pi$ contains a 
point $y$ not on $H$ and not on the line $\<x_1,x_2\>$. It follows that $y$ lies on a line $M$ 
meeting both $L_1$ and $L_2$ in points, say $z_1,z_2$, respectively, not in $\<H\>$. Since 
$y\in  [z_1,z_2]\cap \<H\>$, and $[z_1,z_2]\neq \<H\>$, this contradicts (S2). 
\end{proof} 

The next lemma is the last one before we invoke (S3*). It basically says that $N\geq 7$.

\begin{lemma} \label{lemma5bis} Suppose $X$ does not contain singular planes. Let $H$ be a 
hypo and $L$ a singular line on $H$. Let $H_1$ and $H_2$ be two different hypos distinct from 
$H$ containing $L$. Then $W=:\<H,H_1,H_2\>$ has dimension $7$.
\end{lemma}

\begin{proof} 
Since $H,H_1,H_2$ all contain $L$, we already have $\dim W\leq 7$. Suppose now, by way of 
contradiction, that $\dim W\leq 6$. Let $M_i$, $i=1,2$, be a singular line on $H_i$ disjoint from $L$.  Since $W_i=:\<H,H_i\>$ has dimension $5$, $i=1,2$, and since $W=\<W_1,W_2\>$, we see that $W_1\cap W_2$ has dimension at least $4$, hence $4$ or $5$. If it has dimension $5$, then we choose a point $m_1$ on $M_1$ and we set $U=\<H,m_1\>$. If it has dimension $4$, then we put $U=W_1 \cap W_2$. In both cases, $U$ has codimension $1$ in $\<H_i\>$, for $i\in\{1,2\}$. Since $M_i$, $i=1,2$, does not meet $\<H\>$, $U$ intersects $M_i$ in exactly one point, which we may denote $m_i$.    If we denote the generator of $H_i$ through $m_i$ distinct from $M_i$ by $R_i$, then $R_i$ intersects $L$, and hence $H$. Lemma~\ref{lemma4} implies that $R_1=R_2$. But then Lemma~\ref{lemma1} tells us that $H_1=H_2$, since both contain $L$ and $R_1=R_2$. This contradiction concludes the proof of the lemma.    
\end{proof} 

We can now prove the following two important lemmas.

\begin{lemma}\label{lemma5} Suppose $X$ does not contain singular planes, and assume that $(X,\Xi)$ satisfies \emph{(S3*)}. Let $H$ be a hypo and $L$ a singular line on $H$. Let $x_1$ and $x_2$ be two distinct points on $L$.
Then the $5$-spaces $U_1=\<H,T_{x_1}\>$ and $U_2=\<H,T_{x_2}\>$ meet in $\<H\>$.
\end{lemma}

\begin{proof} Let $L_1$ and $L_1'$ be two singular lines through $x_1$ not in $H$. By 
Lemma~\ref{lemma3}, $T_{x_1}$ is generated by $L_1,L_1'$ and the two singular lines of $H$ 
passing through $x_1$. The unique hypos $G_{1}$ and $G_{1}'$ containing $L,L_1$ and $L,L_1'$, 
respectively, contain some lines $L_2$ and $L_2'$, respectively, through $x_2$ distinct from $L$. 
Clearly $\<H,G_1,G_1'\>=\<U_1,U_2\>$, which is, by Lemma~\ref{lemma5bis}, $7$-dimensional. It 
follows that $U_1\cap U_2$ is $3$-dimensional, and hence coincides with $\<H\>$. 
\end{proof} 

A natural property to investigate is whether the tangent space $T_x$ at some point $x\in X$ contains points of $X$ not on a singular line with $x$. The following lemma will imply that this is not the case.  

\begin{lemma} \label{lemmatangent}Suppose $X$ does not contain singular planes, and assume 
that $(X,\Xi)$ satisfies \emph{(S3*)}. Let $x\in X$ be arbitrary, and let $H$ be an arbitrary hypo containing $x$. Then all points of $\<H,T_x\>$ in $X$ either lie on $H$, or are on a singular line together with $x$. In particular, all points of $T_x\cap X$ are 
contained in a singular line with $x$.
\end{lemma}

\begin{proof}
Suppose by way of contradiction that some point $y\notin H$, for which $\<x,y\>$ is not a singular line,
is contained in $\<H,T_x\>$. Let $H'$ be the unique hypo through $x$ and $y$. Then 
$H'\subseteq\<H,T_x\>$, and so, since the latter is $5$-dimensional, $\<H'\>\cap\<H\>$ is $1$-dimensional. This implies that $H'$ and $H$ share some singular line $L$ through $x$, and 
hence there is a singular line $M$ containing $y$ and a point $z$ of $L$. By assumption, 
$z\neq x$. But the space $\<H,T_x,M\>$ is $5$-dimensional, hence, since $M\subseteq T_z$ and 
$M$ is not in $\<H\>$, this implies that the dimension of $\<H,T_x,T_z\>$ is at most $6$, 
contradicting Lemma~\ref{lemma5bis}.
\end{proof}  

\begin{lemma} \label{lemmaN=8}Suppose $X$ does not contain singular planes, and assume 
that $(X,\Xi)$ satisfies \emph{(S3*)}. Then $N=8$.
\end{lemma}

\begin{proof}
We first show that $N>7$. Indeed, suppose $N=7$. Fix an arbitrary hypo $H$ and two distinct 
points $x_1,x_2$ on a singular line $L$ of $H$. Choose two singular lines $L_1,L_1'$ through 
$x_1$ outside $\<H\>$. Let $G_1$ be the hypo through $L_1,L_1'$. Then $\<H,G_1\>$ is $6$-
dimensional. Let $S$ be a $3$-space skew to $\<H\>$. Then the projection of 
$G_1\setminus(L_1\cup L_1')$ from $\<H\>$ onto $S$ is an affine part $\pi^*$ of some plane $\pi$ 
of $S$. The projection of $T_{x_2}$ is a line $T_2$ in $S$. Noting that the line $\pi\setminus\pi^*$ 
is the projection of $T_{x_1}$, it follows from Lemma~\ref{cor1} that $T_2$ intersects $\pi$ in a 
point $z'$ of $\pi^*$. Let $z$ be the point of $G_1$ projected onto $z'$. Then $z$ lies in the $5$-
space $\<H,T_{x_2}\>$, contradicting Lemma~\ref{lemmatangent}.

Now we show that $N\leq 8$. Indeed, suppose by way of contradiction that $N>8$. Let $H$ again be an arbitrary hypo, and let $x_1,x_2,L$ be as above. Since $W=:\<H,T_{x_1},T_{x_2}\>$ is $7$-dimensional, by Lemma~\ref{cor1}, we can choose two points $y,z\in X$ such that the line $\<y,z\>$ is skew to $W$. Then any hyperbolic space containing $y$ and $z$ has at most a line in common with $W$, implying that we find a singular line $R$ skew to $W$. We claim that there can be at most one singular line $R_1$ connecting a point of $R$ with a point of $M_i$, where $M_i$ is the singular line of $H$ through $x_i$ distinct from $L$, for $i=1,2$. Indeed, if there were at least two such lines, then Lemma~\ref{lemma1} would imply that $R$ and $M_i$ lie in some hypo together, and so some point of $R$ would be contained in $T_{x_i}$, a contradiction to the choice of $R$. Hence we can find a point $v$ on $R$ not on any singular line meeting $M_1\cup M_2$ nontrivially.  This implies that the hypos $X([v,x_1])$ and $X([v,x_2])$ meet $H$ in a point. Consequently, $T_v$ meets $T_{x_i}$ in a line, for $i=1,2$, and these lines are skew as they do not meet $\<H\>\supseteq T_{x_1}\cap T_{x_2}$. Hence $T_v$ intersects $W$ in a $3$-space $W_v$. But $T_v$ also contains $R$ , and our choice of $R$ implies that $\<W_v,R\>$ is $5$-dimensional, contradicting Axiom~(S3*). 
\end{proof}

We can now show that there is at least one singular plane.

\begin{lemma}\label{lemma6} Assume 
that $(X,\Xi)$ satisfies \emph{(S3*)}. Then there exists a singular plane in $X$.
\end{lemma}

\begin{proof} Suppose on the contrary that there is no singular plane in $X$. Let $H$ be an arbitrary hypo, and, as above, let $x_1,x_2$ be two points on $H$ on a common singular line $L$.
Let $L_1,L_1'$ be two distinct singular lines through $x_1$, not inside $H$. Let $H_1$ be the hypo through $L_1$ and $L$ and let $L_2$ be the singular line of $H_1$ not in $H$ and passing through $x_2$. Finally, let $L_2''$ be an arbitrary singular line through $x_2$, distinct from $L_2$ and not in $H$. We denote the hypo through $L_1$ and $L_1'$ by $G_1$. As before, we have 
$\<G_1\>\cap\<H\>=\{x\}$. By Lemma~\ref{lemmaN=8}, we have $N=8$ and so we can choose a $4$-dimensional subspace $F$ skew to $\<H\>$. With ``projection'', we mean in this proof the projection from $\<H\>$ onto $F$. As before, the projection of $\<G_1\setminus(L_1\cup L_1')\>$ is a set of 
points  of an affine plane $\pi^*_1$ in $F$, with projective completion $\pi_1$. Likewise the projection of the hypo $G_2$ through 
$L_2,L_2''$, except for $L_2$ and $L_2'$, is an affine plane $\pi^*_2$ in $F$, with projective completion $\pi_2$. The projective 
planes $\pi_1$ and $\pi_2$ meet by a dimension argument. But by Lemma~\ref{lemma5}, the intersection $(\pi_1\setminus\pi^*_1)\cap(\pi_2\setminus\pi_2^*)$ is empty. 

Suppose now that a point of $\pi_1$ coincides with a point at infinity of $\pi_2$ (so belonging to $\pi_2\setminus\pi_2^*$). This means that a 
point $z$ of $G_1$ is contained in $\<H,T_{x_2}\>$, contradicting Lemma~\ref{lemmatangent}. 

Hence, by symmetry, the affine planes $\pi^*_1$ and $\pi_2^*$ meet in a unique (affine) point and so we 
have points $z_1$ in $G_1$ and $z_2$ in $G_2$ lying in a common $4$-space $S$ with $H$. We now prove that $z_1=z_2$. To that aim, suppose $z_1\neq z_2$. Let $G$ be a hypo through $z_1,z_2$.
Considering $G\cap\<H\>$ and (S2), we see that $\<z_1,z_2\>$ is a singular line hitting $H$ in some 
point $u$. By possibly interchanging the roles of $x_1$ and $x_2$, we may assume that $\<u,x_2\>$ is not a singular line ($u\notin L$ because otherwise $L$ belongs to $G_1$ since by Lemma~\ref{lemma1}, there is a (unique) hypo through $z_1,x_1,L$).  Note that $z_1$ does not belong to $T_{x_2}$ (as $z$ does not belong to the projection $\pi_2\setminus\pi_2^*$ of $T_{x_2}$). So we can consider the unique hyperbolic space $[z_1,x_2]$ and the corresponding hypo $G_2'$. 
Suppose that $G_2'$ contains a singular line of $H$. Then we can find a quadrangle of singular lines 
containing $z_1, u$ and two singular lines of $H$. By Lemma~\ref{lemma1}, this  implies that $z_1$ belongs to $H$, a 
contradiction. Hence the two singular lines of $G_2'$ through $x_2$ do not lie in $\<H\>$ and 
consequently, the projection from $\<H\>$ of $[z_1,x_2]\setminus\{x_2\}$ coincides with the plane $\pi_2$. Now, the arguments in the previous paragraph imply that $\pi_1$ and $\pi_2$ span 
$F$, and this implies that the singular lines in $G_1$ and $G_2'$ through $z_1$ span a $4$-dimensional space which must necessarily coincide with $T_{z_1}$, and which is  projected onto $F$.  Consequently $T_{z_1}$ is disjoint from $\<H\>$, contradicting $u\in T_{z_1}$. 

Hence we have shown that $z_1=z_2$. Now let $M_i$ be the singular line in $G_i$ meeting $L_i$, $i=1,2$. Note that $M_1\neq M_2$ as otherwise the quadrangle $M_1,L_1,L,L_2$ and Lemma~\ref{lemma1} imply that $G_1=G_2$. Remember that $L_1,L_2$ are contained in the hypo $H_1$, and so there is a singular line intersecting $L_1$ and containing the intersection point of $L_2$ and $M_2$. This gives us again a quadrangle of singular lines (it cannot degenerate to a triangle, as this contradicts the hypothesis that we do not have singular planes). Again, Lemma~\ref{lemma1} implies that $z_1$ belongs to $H_1$, a contradiction. 

This contradiction finally implies that a singular plane exists. 
\end{proof}

Next we show that if there is a \emph{singular $3$-space} in $X$, we obtain $\cS_{1,3}$. 

\begin{lemma}\label{lemma9} Assume 
that $(X,\Xi)$ satisfies \emph{(S3*)}. If there is a $3$-space entirely contained in $X$, then $X$ is $\cS_{1,3}$.
\end{lemma}

\begin{proof}
Suppose that some $3$-space $S$ is entirely contained in $X$. Let $x$ be any 
point of $S$ and let $H$ be a hypo through $x$. Then $H$ cannot contain two lines of $S$. If none 
of the singular lines $L_1,L_2$ of $H$ through $x$ belongs to $S$, then (S2) implies that the 
intersection of the plane $\<L_1,L_2\>$ with $S$, which is a line by (S3), belongs to $H$, a 
contradiction.  Hence exactly one of $L_1,L_2$ belongs to $S$. 

In fact, there is a unique singular line $L_x$ through $x$ not in $S$. Indeed, in the previous 
paragraph we showed that there was at least one. If there were at least two, then the plane 
generated by these would contain a third one (in $S$), and so this would be a singular plane 
$\pi$. If some point $p$ in $T_x$ did not belong to $X$, then we find two intersecting lines $L_S$ 
and $L_\pi$ in $S$ and $\pi$, respectively, such that $p\in\<L_S,L_\pi\>$. As in the proof of 
Lemma~\ref{lemma1}, this implies that $L_S\cup L_\pi$ is contained in a unique hypo $H_L$. We 
now consider another pair of intersecting lines $M_S,M_\pi$, in $S$ and $\pi$, respectively, such 
that $L_S\cap L_\pi\neq M_S\cap M_\pi$, and such that $p\in\<M_S,M_\pi\>$. Then $M_S$ and 
$M_\pi$ are contained in a unique hypo $H_M\neq H_L$, implying that $p$, which belongs to 
$\<H_L\>\cap\<H_M\>$, belongs to $X$ after all. Hence we see that this would imply that all lines 
of $T_x$ are singular, a contradiction (this would imply that there are no hypos through $x$). 

Considering the hypo through any point $y\in X\setminus S$ and any point $x\in S$, we see that 
every point $y$ of $X$ not in $S$ is on a (necessarily unique) singular line that intersects $X$. 

Now let $y_1,y_2$ be two points of $X\setminus S$, and let $x_1,x_2$ be the unique points of 
$S$ on a singular line with $y_1,y_2$, respectively. Suppose that $\<y_1,y_2\>$ is not a singular 
line. Since the hypo containing $y_1,x_2$ automatically contains the singular line through $x_2$ 
not contained in $S$, it also contains $y_2$ and hence it coincides with the hypo through 
$y_1,y_2$. We have shown that all hypos have a unique line in $S$. This immediately implies that 
there is no hypo containing two intersecting singular lines not meeting $S$. The same argument 
as in the proof of Lemma~\ref{lemma1} implies that all singular lines through a point 
$y\in X\setminus S$ distinct from the one intersecting $S$ in a point, form a singular subspace 
$S_y$, which is thus defined by any of its points. It follows that, using the hypos through a point of 
$S$ and a point of $S_y$, there is a (bijective) collineation between $S$ and $S_y$ given by 
``lying on a common singular line'', and hence $S_y$ has dimension $3$. Also, all points of $X$ 
are on lines intersecting both $S$ and $S_y$ (we can see this also by letting $S_y$ play the role 
of $S$). Since $S$ and $S_y$ generate a space of dimension seven, \cite{Zan:96} implies that we obtain $S(1,3)$.
\end{proof}

From now on we assume that $X$ does not contain a singular 3-space so that $X$ is not a Segre variety 
of type $(1,3)$. 

Our next goal is to single out $\cS_{1,2}$ and to show that otherwise every point of a singular plane belongs to a second singular plane. 
We first show that, if $X$ is not $\cS_{1,2}$, then there must be a lot of other singular lines through a point of such singular plane 
besides the ones in that singular plane.

\begin{lemma}\label{lemma7} Assume 
that $(X,\Xi)$ satisfies \emph{(S3*)}, and let $\pi$ be a singular plane such that 
$x\in \pi$. Then either there are at least three singular lines through $x$ not contained in $\pi$, or $N=5$ and $X$ is projectively equivalent to $\cS_{1,2}$. 
\end{lemma} 

\begin{proof} Considering an arbitrary hypo through $x$ we see that there must be at least one 
singular line through $x$ not in $\pi$. Suppose now that there is only one such, namely $L_x$. 
Suppose that there are at least two singular lines $L_y,L_y'$ not in $\pi$ but through some other 
point $y\in \pi$. Then there must be hypos $H$ containing $L_y$ and $\<x,y\>$ and $H'$ 
containing $L_y'$ and $\<x,y\>$. These must automatically contain $L_x$, but the intersection of 
the corresponding hyperbolic spaces is then a plane, a contradiction. Hence through every point 
of $\pi$ there is a unique singular line not contained in $\pi$. As in the proof of 
Lemma~\ref{lemma9}, we deduce that $X$ is contained in the space generated by two disjoint 
planes. Similarly to the last part of the proof of Lemma~\ref{lemma9}, we obtain the Segre variety $\cS(1,2)$.

Now suppose that there are exactly two lines $L_x,L_x'$ through $x$ not in $P$. Then there is a 
hypo $H_x$ containing these. 



Consider three points $x_{1},x_{2},x_{3}$ on a line $M$ in $\pi$ not through $x$, and a point $t$ on $L_{x}$. Consider the hypos $H_{1},H_{2}$ and $H_{3}$ through $L_{x}$ and $x_{1},x_{2},x_{3}$, respectively. Denote the lines through $t$ different from $L_{x}$ and contained in these hypos by $M_{1},M_{2}$ and $M_{3}$ respectively. The points $x_{1},x_{2}$ and $x_{3}$ are collinear with points $t_{1},t_{2}$ and $t_{3}$ on $M_{1},M_{2}$ and $M_{3}$ respectively. 

Suppose first that $M_{1},M_{2}$ and $M_{3}$ are not coplanar. Now consider the hypos $H'_{1},H'_{2}$ and $H'_{3}$ determined by $M$ and $t_{1},t_{2}$ and $t_{3}$, respectively. Since $t_{1},t_{2}$ and $t_{3}$ are not collinear we obtain three different lines $L_{1},L_{2}$ and $L_{3}$ through $x_{2}$ outside of $\pi$. Now looking at hypos determined by $L_{i}$ and $x$ yields three different lines through $x$, a contradiction.



Hence $M_{1}$ and $M_{2}$ are contained in a singular plane. Now choose a point 
$z$ on $L_x'$ and a point $z'$ on $M_{1}$, but not on $L_x$. If $\<z,z'\>$ were singular, then the 
quadrangle $L_x,L_x',\<z,z'\>,M_{1}$ would show, using Lemma~\ref{lemma1}(2), that $M_{1}$ belongs to 
$H_x$, a contradiction. Let $H_z$ be the hypo through $z$ and $z'$. Let $L_z$ be a singular line 
of $H_z$ through $z$. Suppose that $L_z$ is contained in $H_x$. Let $z''$ be the unique point on 
$L_z$ such that $\<z',z''\>$ is singular, and let $z'''$ be the unique point on $L_z$ such that 
$\<t,z'''\>$ is singular. Considering the quadrangle $L_z,\<t,z'''\>,M_{1},\<z',z''\>$, it follows from 
Lemma~\ref{lemma1}(2) that $M_{1}$ belongs to $H_x$, a contradiction. Hence no singular line 
through $z$ in $H_z$ belongs to $H_x$. But now interchanging the roles of $L_x$ and $L_x'$, 
and of $t$ and $z$, we arrive at a contradiction as in the previous paragraph.

The lemma is proved.
\end{proof}

From now one, we assume that through any point $x$ of any singular plane $\pi$, there are at least three singular lines not contained in $\pi$.
We can now show the penultimate step.

\begin{lemma}\label{lemma8} Assume 
that $(X,\Xi)$ satisfies \emph{(S3*)}. Let $\pi$ be a singular plane and 
$x\in \pi$. Then there is a second singular plane through $x$. The union of both planes contains all 
singular lines through $x$. 
\end{lemma} 

\begin{proof} By the previous lemma there are at least three singular lines through $x$ not in 
$\pi$, name them $L_1,L_2,L_3$. If they are contained in a plane, then this plane is singular. If 
they are not contained in a plane, then the $3$-space they generate contains a line $L_4$ of $\pi$ 
(using (S3*)). If no pair of $\{L_1,L_2,L_3\}$ is contained in a singular plane, then the planes 
$\<L_1,L_2\>$ and $\<L_3,L_4\>$ are distinct and hence meet in line $L_5$. 
Lemma~\ref{lemma1}(1) implies that $L_5$ is a singular line, and therefore, $\<L_3,L_4\>$ is 
singular after all. 

So we always have at least two singular planes. If another singular line through $x$ existed (not contained in either singular plane), then it would be contained in a plane together with a line of each singular plane. This now leads to a singular 4-space, a contradiction. 
\end{proof}

Now we can finish the proof of Theorem~\ref{theo1}. 

It follows from the previous lemmas that through every point of the two singular planes through $x$ 
there are precisely two singular planes. By connectivity, this is true for all points of $X$. Also, 
consider two disjoint singular planes. By considering appropriate hypos, we see that every 
singular plane intersecting one of them, intersects the other. We can use this to see that, starting 
from two intersecting singular planes $\pi,\pi'$, the set $X$ is the disjoint union of all singular 
planes meeting $\pi$ and likewise for those meeting $\pi'$. Moreover, as in the proof of Lemma~\ref{lemma9}, the hypos define 
collineations between disjoint planes and this is enough to conclude with \cite{Zan:96} that we have a Segre variety of type $(2,2)$.


\section{Hjelmslevian Veronesean Sets} \label{sec:HVS}
In this section, we consider the Mazzocca-Melone axioms for the class of ``tubes''. 
 
A \emph{tube} $C$ (as a short synonym for ``cylinder'') in a $3$-dimensional projective space $\Sigma$ (over $\K$) is the set of points of 
$\Sigma$ on some nondegenerate cone with base a plane oval $O$, except for the vertex $v$ ($v$ will also be called the \emph{vertex} of the tube, although it does never belong to the tube).  For every point $x\in C$, there is a unique plane $\pi$ 
through $x$ intersecting $C$ in the unique generator which contains $x$. The plane $\pi$ 
contains all lines through $x$ that meet $C$ in only $x$ and is called the \emph{tangent plane} at 
$x$ to $C$ and denoted $T_x(C)$. 

Let $X$ be a spanning point set of $\mathrm{PG}(N,\mathbb{K})$, $N>3$, and let $\Xi$ be a 
collection of 3-dimensional projective subspaces of $\mathrm{PG}(N,\mathbb{K})$, called the 
\emph{cylindric spaces} of $X$, such that, for any $\xi\in\Xi$, the intersection $\xi\cap X$ is a tube 
$X(\xi)$ in $\xi$ (and then, for $x\in X(\xi)$, we sometimes denote $T_x(X(\xi))$ simply by 
$T_x(\xi)$). Conversely, for a tube $C$, we denote $\Xi(C)$ the unique member of $\Xi$ containing $C$.  We denote by $Y$ the set of all vertices and we obviously have $X\cap Y=\emptyset$.  
We call $(X,\Xi)$, or briefly $X$, a \emph{Hjelmslevian Veronesean set (of index $2$)} if the following properties hold :

\begin{itemize}
\item[(H1)] Any two points $x$ and $y$ lie in at least one element of $\Xi$, which we denote by $[x,y]$ if it is unique.

\item[(H2)] If $\xi_1,\xi_2\in \Pi$, with $\xi_1\neq \xi_2$, then $\xi_1\cap\xi_2\subset X\cup Y$, and $\xi_1\cap\xi_2\cap Y$ is contained in a codimension 1 subspace of $\xi_1\cap\xi_2$. If $Y\cap\xi_1\cap\xi_2\neq\emptyset$, then $Y\cap\xi_1\neq\emptyset\neq Y\cap\xi_2$.

 


\item[(H3*)] For each $x\in X$, all planes $T_x([x,y])$, $y\in X\setminus\{x\}$, are contained in a 
fixed 4-dimensional subspace of $\mathrm{PG}(N,\mathbb{K})$, denoted by $T_x$.
 \end{itemize}

We again have that (H1) implies that $|\Xi|>1$. 

The rest of this section is devoted to prove the following theorem.

\begin{theorem} \label{theo2} Either $N=6$ or $N=8$. In either case, the point set $X$ of a Hjelmslevian Veronesean set of index $2$ is projectively unique.
\end{theorem}

A {\em singular line} $L$ is a generator of a tube; if $|\K|>2$, then $L$ is obviously characterized by saying that it contains at least three points of $X$, and then all of its points belong to $X$, except for one, which belongs to $Y$, and which we denote by $y(L)$. The set $L\setminus y(L)$ is sometimes referred to as a \emph{singular affine line}.


The case $|\K|=2$ will be treated separately at the end. So, for most of the results below, we assume $|\K|>2$, although for the moment, we allow $|\K|=2$.

\begin{lemma}\label{lemmaH0} Let $C,C'$ be two tubes with common vertex. Then $C$ and $C'$ share a unique generator.
\end{lemma}
\begin{proof} By Axiom (H2), the intersection $\Xi(C)\cap \Xi(C')$ is either a generator of both $C$ and $C'$ (in which case there is nothing to prove), or it is a point $t$ of $X\cup Y$. In our assumptions, $t$ is the vertex, and so belongs to $Y$. But this contradicts (H2) as a codimension 1 subspace of a point is the empty set.  
\end{proof}

\begin{lemma}\label{lemmaH1} Let $C,C'$ be two tubes and suppose $|\Xi(C)\cap \Xi(C')|>1$. Then $C\cap C'$ is a singular affine line. In particular, $C$ and $C'$ have the same vertex.
\end{lemma}
\begin{proof}
From (H2) it follows that $\Xi(C)\cap \Xi(C')$ is a generator $L$ of both $C$ and $C'$. It suffices to show that the vertices of $C$ and $C'$ coincide. Let $t$ and $t'$ be these respective vertices, and assume $t\neq t'$. Then both $t$ and $t'$ belong to $X\cap Y$, a contradiction.
\end{proof}

A plane $\pi$ in $\PG(N,\K)$ will be called \emph{singular} if it contains a unique line $L$ all of whose points belong to $Y$, and all other points of $\pi$ belong to $X$. The line $L$ will be referred to as the \emph{radical line} of $\pi$.

\begin{lemma}\label{lemmaH1'}
Let $L$ and $L'$ be two singular lines intersecting in a point $y\in Y$. Then either $L$ and $L'$ belong to a unique common tube, or $\<L,L'\>$ is singular.
\end{lemma}
\begin{proof}
Suppose $L$ and $L'$ do not belong to a unique common tube. Take a point $p$ in the plane $\pi$ spanned by $L$ and $L'$ arbitrarily, but not on $L\cup L'$. We consider two lines $M,M'$ in $\pi$ through $p$ not incident with $y$. These lines intersect $L$ and $L'$, respectively, in the points $x_{LM},x_{LM'}$ and $x_{L'M}$ and $x_{L'M'}$, respectively, using self-explaining notation. Arbitrary tubes $C$ and $C'$ through $x_{LM},x_{L'M}$ and $x_{LM'},x_{L',M'}$, respectively, satisfy $p\in\Xi(C)\cap\Xi(C')$. Now note that, if $C=C'$, then $C$ contains $L$ and $L'$, a contradiction. Hence $C\neq C'$. By (H2), $p$ belongs to $X\cup Y$. It is impossible that $p$ belongs to $X$ for all legal choices of $p$, since this would imply that there are lines all of whose points are contained in $X$. Hence we may assume that $p$ belongs to $Y$ (if $|\K|=2$, then this is automatic, and we have a singular plane; so from now on we may assume $|\K|>2$). Then all other points of $M\cup M'$ belong to $X$. The line $N=\<p,y\>$ contains two elements of $Y$; hence we claim it can contain at most one element of $X$.  Indeed, if it contained two points of $X$, then it would be contained in some member of $\Xi$, which contains only one element of $Y$, by Lemmas~\ref{lemmaH0} and~\ref{lemmaH1}. Suppose now that $N$ contains a unique point $x\in X$. Now note that every point $z$ of $\<L,L'\>\setminus N$ belongs to $X$, as the line $\<p,z\>$ contains at least two points of $X$, namely the intersections with $L$ and $L'$. So any line through $x$ in $\<L,L'\>$ only contains points of $X$, a contradiction. Hence $\<L,L'\>$ is a singular plane.  
\end{proof}

\begin{lemma}\label{lemmaH2} Let $C$ be a tube and $c\in C$. Then either $N=6$ and $X\cup Y$ is projectively equivalent to a cone with vertex some point $t$ and base a quadric Veronese variety of $\PG(2,\K)$, with $Y=\{t\}$, or there exists some tube $C'$ intersecting $C$ in the singleton $\{c\}$. 
\end{lemma}
\begin{proof}
Suppose that all tubes whose intersection with $C$ contains $c$ have a generator in common with $C$. Let $x\in X$ be arbitrary. 
Then any tube through $x$ and $c$ must have the same vertex $t$ as $C$, by Lemma~\ref{lemmaH1}. Hence $\<x,t\>$ is a singular line. 
Now we project $X$ from $t$ onto some hyperplane $H$ not containing $x$ and denote the projection operator by $\rho$. We 
define the following geometry $\mathcal{G}$. The point set is $\rho(X)$, and the line set is the family of projections of tubes 
with vertex $t$. In $H$, the points of $\mathcal{G}$ are projective points, and the lines of $\mathcal{G}$ are planar 
ovals.  By Lemma~\ref{lemmaH0}, every pair of distinct lines of $\mathcal{G}$ has a unique point in common. Now let $L$ and $L'$ be 
two singular lines through $t$. Suppose that $\<L,L'\>$ is singular. Let $C'$ be any tube containing $L$. Choose a singular 
line $M$ in $\<L,L,'\>$ not though $t$ and a tube $D$ through $M$. Let $z$ be a point of $D$ not on $M$. Since by the beginning of 
our proof, $K=:\<t,z\>$ is a singular line, we have, by Lemma~\ref{lemmaH1'}, that either $\<K,L\>$ is a singular plane, or 
$K,L$ are contained in a unique tube $D'$. In the latter case, $\Xi(D)\cap\Xi(D')$ violates (H2), because it contains the two 
points $z$ and $z'=:L\cap M$ not on a singular line. In the former case, the line $\<z,z'\>$ is singular, again a contradiction. 

So we conclude that $\<L,L'\>$ is not singular, and so, by Lemma~\ref{lemmaH1'},  there is a unique tube containing both of them. 
This implies that every pair of points of $\mathcal{G}$ is contained in a unique line of $\mathcal{G}$. Hence $\mathcal{G}$ 
is a projective plane (since lines have at least three points). Since the $3$-spaces of two different tubes intersect in a line, we deduce that the two planes generated by the two ovals corresponding to two lines of $\mathcal{G}$ meet in a unique point, which implies that the dimension of $H$ is at least $4$. We claim that $H$ has dimension $5$. Indeed, by the second main result of \cite{Sch-Mal:**}, the dimension of $H$ is at most $5$. So assume that $H$ has dimension $4$. Let $D$ be any oval corresponding to a line of $\mathcal{G}$, and let $x\in D$. Let $D_1,D_2$ be two distinct ovals corresponding to lines of $\mathcal{G}$ containing $x$, and distinct from $D$. Let $\zeta$ be a $3$-dimensional space containing $D$ and not containing the tangent lines at $x$ of $D_1$ and $D_2$ ($\zeta$ exists since there are at least three $3$-spaces through the plane $\<D\>$). Then $\zeta$ intersects $D_1$ and $D_2$ in distinct points $x_1,x_2$, respectively. There is a unique oval $D'$ corresponding to a line of $\mathcal{G}$ containing $x_1,x_2$, and since $x_1,x_2$ belongs to $\zeta$, the planes $\<D\>$ and $\<D'\>$ intersect in a point of $\<x_1,x_2\>$, which is a point of $\mathcal{G}$ by Axiom~(H3*). But then $D'$ contains three points on a projective line, a contradiction. Our claim is proved. 

Now, again by the second main result of \cite{Sch-Mal:**}, $\mathcal{G}$ is isomorphic to the quadric Veronese variety of the projective plane $\PG(3,\K)$. Adding $t$, we see that $N=6$. 
\end{proof}

From now on we assume $|\K|>2$ and we may also assume that $X$ is not a point over a quadratic Veronese variety so that the second conclusion of Lemma 3.5 holds.

\begin{lemma}\label{lemmaH3} Two singular lines $L,L'$ intersecting in a point $x\in X$ generate a singular plane.
\end{lemma}
\begin{proof}
Let $t$ and $t'$ be the members of $Y$ on $L,L'$, respectively. Then every point $z$ on $\<t,t'\>$ except for $t$ and $t'$ themselves is the intersection of two lines meeting $(L\cup L')\setminus\{t,t'\}$ in two points. Since $L\cup L'$ cannot be contained in a tube (since that tube would have two vertices $t$ and $t'$), these two lines are contained in distinct tubes. Hence $z\in X\cup Y$. But at most one point on $\<t,t'\>$ can belong to $X$, as otherwise a tube through $\<t,t'\>$ would have at least two vertices. So, there exists a point $z_0\in Y$ on $\<t,t'\>$ contained in two singular lines of $\<L,L'\>$. Noting that the plane $\<L,L'\>$ contains already at least four distinct singular lines, the assertion follows from Lemma~\ref{lemmaH1'}.
\end{proof}

\begin{lemma}\label{lemmaH4} Let $x\in X$ be arbitrary. Then the set of points of $X$ on a common singular line with $x$ is either an affine plane, or an affine $3$-space. 
\end{lemma}
\begin{proof}
Let $V_x$ be the set of points of $X$ on a common singular line with $x$. Lemma~\ref{lemmaH2} tells us that there are at least two singular lines $L,L'$ through $x$. Hence, by Lemma~\ref{lemmaH3}, the affine plane consisting of the points of $X$ in $\<L,L'\>$ belongs to $V_x$.  If there is some other singular line $L''$ through $x$ not contained in $\<L,L'\>$, then a repeated use of Lemma~\ref{lemmaH3} shows that the set of points of the $3$-space $\<L,L',L''\>$ belonging to $X$ (and on a singular line with $x$) is an affine $3$-space. A similar argument shows that, if there was still another singular line not contained in $\<L,L',L''\>$, then an affine part of $T_x$ would be contained in $V_x$, but this contradicts the fact that $T_x$ contains planes which intersect $X$ in just an affine line. 

The proof is complete.      
\end{proof}

We will call an affine $3$-space as in the statement of Lemma~\ref{lemmaH4} a \emph{singular $3$-space}. 

Before stating the next proposition, we need a definition. Let $k\leq l;n=k+l+1$ and take complementary subspaces $\Pi,\Pi'$ in $\PG(n,\mathbb{K})$ of dimensions $k$ and $l$ respectively. Now choose normal rational curves $C$ and $C'$ in $\Pi$ and $\Pi'$ respectively and a bijection $\phi:C'\to C$ between them which preserves the cross-ratio. Then we define a \emph{normal rational scroll} $\mathfrak{S}_{k,l}=\cup_{p\in C'} \<p,p^\phi\>$. From now on we call $\mathfrak{S}_{1,2}$ a
{\em normal rational cubic scroll}.

\begin{prop}\label{propH1} Let $y\in Y$ and let $X_y$ be the set of all members $x$ of $X$ such 
that $\<x,y\>$ is a singular line. Then $\<X_y\>$ is a $5$-dimensional subspace. Also, $X_y\cup\{y\}$ is a cone with vertex $y$ on a normal rational cubic scroll. In particular, all tubes arise from quadratic cones. Also, for every $x\in X$, the  set of points of $X$ on a common singular line with $x$ is an affine plane.
\end{prop}
\begin{proof}
Let $C$ be a tube with vertex $y$. Pick a generator $L$ on $C$, and a singular plane $\pi$ through $L$, which exists by Lemma~\ref{lemmaH4}. Let $G$ be any generator of $C$ distinct from $L$, and let $M$ be any singular line in $\pi$ through $y$ distinct from $L$. Suppose $\<G,M\>$ is singular. Take any point $t\neq y$ in $\<G,M\>$, with $t\in Y$. Let $G'$ be a line through $t$ meeting $G$ in a point $x'$ of $X$, and let $M'$ be a line through $t$ meeting $L$ in a point $x''$ of $X$.  If $G'\cup M'$ is contained in a cylindric space, then the intersection with $\Xi(C)$, which contains $\<x',x''\>$,  violates (H2), noticing that $x'$ and $x''$ are not contained in a common generator of $C$. Hence, by Lemma~\ref{lemmaH1'}, the plane $\<G',M'\>$ is singular. But that would mean that all points of $\<x',x''\>$ belong to $X\cup Y$, also a contradiction (to $X(\Xi(C))=C$).

Hence $G$ and $M$ define a unique tube $C_{GM}$.    Now we fix two generators $G,G'$ on $C$ distinct from $L$. 
We denote by $U$ an $(N-3)$-space skew to the plane $\alpha=:\<G,G'\>$, and we let $\rho$ be the projection operator from $\alpha$ onto $U$. With ``projection'', we always mean $\rho$ in the rest of this proof, except for the last paragraph.  

The projection of $C_{GM}$, for any line $M$ in $\pi$ through $y$ distinct from $L$, is an affine part of a line $L_M$; we denote the projection of the tangent plane to $C_{GM}$ at any point of $G\cap X$ by $p_M$ (it is the point at infinity of $L_M$ with regard to the affine part alluded to above). Likewise, the projection of $C_{G'M}$ is a line $L'_M$ minus a point $p'_M$. All these lines have a point in common with the projection of $\pi$, which is a line $L_\pi$. As two tubes never belong to the same $4$-space due to (H2), we see that $L_\pi$ differs from any such $L_M$ and any such $L'_M$, and no $L_M$ coincides with any $L'_{M'}$. Moreover, intersecting $L_\pi$ defines a bijection from $\{L_M: M\mbox{ a line in }\pi\mbox{ through }y\mbox{ distinct from }L\}$ into $L_\pi\setminus\{\rho(L)\}$. All this now easily implies that all $L_M$ and all $L'_M$ above are contained in a unique plane $\beta$.  

Now suppose that $L$ is contained in a singular $3$-space $S$. One easily checks that $\rho(S)$ cannot coincide with $\beta$. Consider a line $N$ through $y$ in $S$ projecting outside $\beta$. As above, there is a unique tube $C_{GN}$ containing $G$ and $N$. This tube has to intersect all tubes $C_{G'M}$, with $M$ a line in $\pi$ distinct from $L$ and incident with $y$, in distinct generators.  Since $\rho$ is injective on the set of generators of $C_{GN}$ distinct from $G$, this easily implies that the projection of $C_{GN}$ is contained in $\beta$, and hence $\rho(N)$ belongs to $\beta$, contradicting our choice. 

We have shown that singular $3$-spaces do not exist. Hence Lemma~\ref{lemmaH4} implies that every singular line is contained in a unique singular plane. Denote the singular planes through $G$ and $G'$ by $\pi_G$ and $\pi_{G'}$, respectively.

Now let $P$ be any singular line through $y$, $P\notin\{G,G'\}$. Then $P$ cannot be contained in  both $\pi_G$ and $\pi_{G'}$; suppose it is not contained in $\pi_G$. Then $P$ and $G$ are contained in a unique tube $C_{PG}$, and the arguments in the paragraph preceding the previous one show that the projection of this tube is contained in $\beta$. Hence $\<X_y\>$ is a $5$-dimensional subspace, as it coincides with $\<\alpha,\beta\>$. But we can say more. Let $L_{PG}$ be the projection of $C_{PG}\setminus G$. Then $L_{PG}$ and $L_\pi$ have a point $z$ in common. Taking inverse images, we see that we have two cases. First case: there is a generator of $C_{PG}$, which we may take to be $P$, such that the $3$-space $\<\alpha,P\>$ contains a line $M$ of $\pi$ through $y$. If $\<\alpha,P\>\in\Xi$, then $P=L=M$ and so $P\subseteq \pi$. If $\<\alpha,P\>\notin\Xi$, then $M\neq L$. If $P\neq M$, then $P\not\subseteq \pi$ and so $P$ and $M$ determine a unique tube, and (H2) leads to a contradiction. Hence, in this first case, we have shown that $P$ belongs to $\pi$. 

Second case: the tangent plane to $C_{PG}$ at $G$ belongs to a $3$-space $\Sigma$ containing $G,G'$ and some line $M$ of $\pi$ through $y$. Then again (H2) implies that $C_{PG}$ and $C_{G'M}$ share a line in that tangent plane, which must necessarily be $G$, and so $M=L$. But then $\Sigma=\Xi(C)$ and $\Xi(C_{PG})\cap\Xi(C)$ is a plane, contradicting (H2) once again. So the second case cannot occur. 

All in all, we have shown that every tube through $G$ intersects $\pi$ in an affine line. Varying $L$, we conclude that every tube with vertex $y$ intersects every singular plane containing $y$ in a singular affine line.  In other words,  the geometric structure $\mathcal{G}_y$ with point set the set of singular lines through $y$ and line set the set of tubes and singular planes through $y$ is a dual affine plane.

Now let $A$ be a $4$-space in $\<\alpha,\beta\>$ not containing $y$. We now project $X_y$ from $y$ onto $A$. This yields a generalized Veronesean embedding of the projective completion of $\mathcal{G}_y$ in $\<\alpha,\beta\>$, provided we let $y$ play the role of the unique point at infinity. The Main Result|General Version of Section~5 of \cite{Tha-Mal:11} implies that the projection of $X_y$ from $y$ onto $A$ is a normal rational cubic scroll. Hence, since $X_y$ consists of the union of lines through $y$, the intersection of $X_y$ with $A$ is itself a normal rational cubic scroll, finalizing the proof of the proposition.    
\end{proof}

For the record, we explicitly write down a particular thing we proved in the course of the previous proof.

\begin{lemma}\label{lemmaHP} For $y\in Y$, the geometric structure $\mathcal{G}_y$ with point set the set of singular lines through $y$ and line set the set of tubes and singular planes through $y$ is a dual affine plane.\qed
\end{lemma}

There are a few interesting corollaries of Proposition~\ref{propH1}, the first of which does not require a proof since it follows immediately from the fact that for every $x\in X$, the set of points of $X$ on a common singular line with $x$ is an affine plane. 

\begin{cor}\label{corH0} Every singular line $L$ is contained in a unique singular plane, containing all the singular lines that intersect $L$ in a point of $X$. \qed
\end{cor}

\begin{cor}\label{corH1} The set $Y$ is the point set of a plane $\pi_Y$ of $\PG(N,\K)$.
\end{cor}
\begin{proof}
Let $y\in Y$ be arbitrary and let $C$ be a tube with vertex $y$. Then, by Proposition~\ref{propH1}, there is a unique singular plane through every generator of $C$, and the union of the radical lines of all these planes is a plane $\pi_y$. 

Now consider a second point $y'\in Y$, and let $x\in X$ be such that $\<x,y'\>$ is a singular line. We may suppose $y'\notin \pi_y$ and consequently $\<x,y\>$ is not singular. If $c\in C$, then any tube $C'$ containing $x$ and $c$ has a vertex $t$ belonging to $\pi_y$ (as $t$ must belong to the radical line of the singular plane through $c$). We clearly have $y'\in\pi_t$. Considering a tube through a point of $C\setminus\<c,y\>$ and $C'\setminus\<c,t\>$, we see that some point of $\pi_y\setminus\<y,t\>$ belongs to $\pi_{t}$. Since also $y\in\pi_t$, this implies that $\pi_y=\pi_t$. But now $y'\in\pi_t=\pi_y$, a contradiction, and so $y'\in\pi_y$ after all.   

The proof of the corollary is complete.
\end{proof}

We keep $\pi_Y$ of the previous corollary as standard notation.

\begin{cor}\label{corH2} Every line $L$ of $\pi_Y$ is contained in a unique singular plane $\pi_L$. The union of all these planes is precisely $X\cup Y$.
\end{cor}
\begin{proof}
Let $y\in L$. Since, by Proposition~\ref{propH1}, the projection of $\pi_Y\setminus\{y\}$ is a line 
(corresponding to the ``vertex line'' of the normal rational cubic scroll), the line $L$ is the radical line of 
some singular plane $\alpha$. If $L$ were the radical line of some second singular plane 
$\alpha'$, then we can consider a line $M$ in $\alpha$ and a line $M'$ in $\alpha'$, with 
$M\cap M'=\{t\}$, $t\in L$. The plane $\<M,M'\>$ is not singular, as this would lead to a singular 
$3$-space $\<\alpha,\alpha'\>$. Hence there is a tube $C_{MM'}$ containing $M$ and $M'$. Now 
Proposition~\ref{propH1} implies that the projection of $\alpha\setminus\{t\}$ and 
$\alpha'\setminus\{t\}$ from $t$ are disjoint lines on a normal rational cubic scroll, hence 
$\alpha\cap\alpha'=\{t\}$, a contradiction.   

By Lemmas~\ref{lemmaH2} and~\ref{lemmaH3}, every point $x\in X$ is contained in some singular plane, which, by Corollary~\ref{corH1}, meets $\pi_Y$ in some line. 
\end{proof}

We can now determine the geometric structure of $X$.

\begin{prop}\label{propH2} Let $T$ be the set of all tubes. Then $(X,T)$ is a {\em projective Hjelmslev plane of level $2$}. More exactly this means the following: the map $\chi:X\rightarrow \pi^*_Y$, where $\pi_Y^*$ is the plane dual to $\pi_Y$, sending the point $x\in X$ to the radical line of the unique singular plane through $x$, is an epimorphism of $(X,T)$ onto $\pi_Y^*$ enjoying the following properties.
\begin{itemize}
\item[\emph{(Hj1)}] Two points of $X$ are always joined by at least one member of $T$; they are joined by a unique member of $T$ if and only if their images under $\chi$ are distinct. 
\item[\emph{(Hj2)}] Two members of $T$ always intersect in at least one point; they intersect in a unique element of $X$ if and only if their images under $\chi$ are distinct.
\item[\emph{(Hj3)}] The inverse image under $\chi$ of a point, endowed with the intersections with non-disjoint tubes, is an affine plane.
\item[\emph{(Hj4)}] The set of tubes contained in the inverse image under $\chi$ of a line, endowed with all mutual intersections, is an affine plane.
\end{itemize}
\end{prop}

\begin{proof}
Surjectivity of $\chi$ follows from Corollary~\ref{corH2}. Now consider the image $\chi(C)$ of an arbitrary tube $C$. Trivially, the radical lines of all singular planes containing points of $C$ contain the vertex $t$ of $C$. By Proposition~\ref{propH1}, all lines in $\pi_Y$ through $t$ can be obtained this way.  Hence $\chi$ is an epimorphism. Note that the image under $\chi$ of a tube is simply its vertex.  

We now show (Hj1). By definition, two elements of $X$ are contained in a tube. From Proposition~\ref{propH1} we deduce that this tube is not unique if and only if the two elements are contained in a singular line. By the last assertion of Proposition~\ref{propH1} and Corollary~\ref{corH0}, this happens if and only if the two elements are contained in a unique singular plane, hence if and only if their images under $\chi$ coincide.

Now we show (Hj2). Let $C,C'$ be two tubes. If they have the same vertex (hence their images under $\chi$ coincide), then they share a singular affine line (Lemma~\ref{lemmaH0}). Suppose now that they have distinct vertices $t$ and $t'$, respectively.  Let $\pi_{tt'}$ be the unique singular plane containing $t,t'$ (see Corollary~\ref{corH2}). Then Lemma~\ref{lemmaHP} implies that $C$ and $C'$ intersect $\pi_{tt'}$ in singular lines $L$ and $L'$, respectively. Since both lines are contained in a common plane, they intersect in a point (and one easily verifies that the intersection point belongs to $X$, and not to $Y$!). 

Given a point $x$, the affine plane in (Hj3) is the affine plane of Corollary~\ref{corH0} arising from the unique singular plane containing $x$ by removing its radical line.

Finally, (Hj4) follows immediately from Lemma~\ref{lemmaHP}, by dualizing.  
\end{proof}

We say that a tube $C$ and a point $x\in X$ are \emph{neighbors} if $\chi(x)$ is incident with $\chi(C)$. 
We denote by $C(X)$ all neighbors of $C$.

Now we prove the following major step.

\begin{prop}\label{propH3} The set $X$ contains a quadratic Veronese variety $\mathcal{V}$ whose conics are conics on tubes and whose points are in canonical bijection with the singular planes. The $5$-space generated by $\mathcal{V}$ is skew to $\pi_Y$, and $N=8$.
\end{prop}

\begin{proof}
Let us fix a tube $C$ (with vertex $y$), and consider a subspace $F$ complementary to $\Xi(C)$. We consider the projection $\rho$ of $X\setminus C$ from $\Xi(C)$ into $F$. We first 
claim that $\rho$ is injective on $X\setminus C(X)$. Indeed, suppose two points 
$x_1,x_2\in X\setminus C(X)$ are projected onto the same point $a\in F$. Then the $5$-space 
$\<C,a\>$ contains $x_1,x_2$, and hence the line $\<x_1,x_2\>$, which is contained in any 
member of $\Xi$ through $x_1,x_2$, must intersect $\Xi(C)$ in a point $u$ of $X\cup Y$. If $u\in X$, 
then $\<x_1,x_2\>$ is singular and $x_1,x_2$ are neighbors of $C$, a contradiction. Hence 
$u\in Y$, which implies, by Proposition~\ref{propH1}, that $x_1,x_2$ belong to a singular plane 
containing a generator of $C$, leading to the same contradiction. Our claim follows.   

Now consider two tubes $C_1,C_2$, both intersecting $C$ in unique points, and such that 
$|C_1\cap C_2|=1$. Applying $\chi$, we see that $x=:C_1\cap C_2$ is no neighbor of $C$. Let 
$i\in\{1,2\}$ and put $C\cap C_i=x_i$. Since $\<C,C_i\>$ is $6$-dimensional, the projection of 
$C_i\setminus\{x_i\}$ from $\Xi(C)$ is isomorphic to the projection of $C_i\setminus\{x_i\}$ from 
$x_i$ (onto a suitable plane), and hence it is an affine plane $\alpha_i$ with one ``point at infinity'' 
$p_i^\infty$ added (the projection of the generator through $x_i$). By injectivity of $\rho$, 
$\alpha_1\cap\alpha_2=\{\rho(x)\}$ and so $\<\alpha_1,\alpha_2\>$ is a $4$-space inside $F$. 
This implies $N\geq 8$. The line at infinity $L_i^\infty$ corresponds to the tangent plane to $C_i$ at 
$x_i$. Note that, by (H3*), all tubes intersecting $C$ in just $x_i$ will be projected onto planes 
containing $L_i^\infty$. Considering the projections of all tubes connecting $x_1$ with a point of 
$C_2$ not in $C(X)$, we see that $X\setminus C(X)$ 
is projected bijectively into the affine $4$-space $A$ 
in $\<\alpha_1\alpha_2\>$ obtained by deleting the $3$-space 
$A^\infty=:\<L_1^\infty,L_2^\infty\>$. It follows that $N=8$. 

In $\alpha_1$, we see that the lines through $p_1^\infty$ are the projections of the generators of 
$C_1$; all other lines distinct from $L_1^\infty$ are projections of conics on $C_1$ through $x_1$. 
Noting that every pair of points is contained in a tube and every tube containing points of 
$X\setminus X(C)$ intersects $C$ in a unique point, we see that every line of $A$ is the projection 
of either a conic on some tube (and the conic intersects $C$ nontrivially), or a generator of some 
tube intersecting $C$ in a unique point. The latter case happens if and only if the point at infinity of 
the line is contained in the projection of $\Pi_Y\setminus\{y\}$, which is a line $L^\infty$ connecting 
$p_1^\infty$ and $p_2^\infty$.

We take a line $L$ in $A^\infty$ skew to  $L^\infty$, and meeting the non-intersecting  lines 
$L_1^\infty$ and $L_2^\infty$, say in the points $z_1$ and $z_2$, respectively. The set $B_i$ of 
affine points on $\<x,z_i$, $i\in\{1,2\}$, is the projection of a ``pointed'' conic. Hence, translated to 
the Hjelmslev plane $(X,T)$, it corresponds to a set of points mapped bijectively under $\chi$ to an 
affine line $B^*_i$ of $\pi_Y^*$, where the point at infinity of that affine line is incident with 
$\chi(C)$. Let $z$ be any point on $L$, then a similar property holds for the set $B$ of affine points 
of $\<x,z\>$; denote the corresponding affine line of $\pi_Y^*$ with $B^*$. Suppose $B^*=B_i^*$ 
for some $i\in\{1,2\}$. Then there are points $u\in B$ and $u_i\in B_i$, distinct from $z$ 
corresponding to the same point of $\pi_Y^*$. Hence the corresponding points of $X$ lie on a 
singular line, and so, by the above, the line $\<u,u_i\>$ meets $L^\infty$, contradicting the fact that 
$L$ is skew to $L^\infty$. By varying the point $x$ in $\alpha=:\<x,L\>\setminus L$, we see that 
$\alpha$ corresponds to a set $\alpha'$ of points of $(X,T)$ mapped bijectively under $\chi$ to an 
affine part of $\pi_Y^*$, namely all points of $\pi_Y^*$ except those corresponding to the image of 
$C$.    Also, $\alpha'$ defines an affine subplane of $(X,T)$. We now claim that we can extend this 
to a projective subplane in a unique way. 

Indeed, if we interchange the roles of $C$ and $C_1$, then the conics corresponding to those lines 
of $\alpha$ intersecting $\<x,z_1\>$, project again onto affine lines. It is easy to see, 
since $|\K|>2$, see also Lemma 3.1 of \cite{Tha-Mal:11}, that all these affine lines are contained in a  
unique affine plane $\alpha_1$. Also, if two lines of $\alpha$ intersect $\<x,z_1\>$ in the same 
point, then it is obvious that the corresponding lines    in $\alpha_1$ are parallel, and so the line 
$\<x,z_1\>$ of $\alpha$ corresponds to the line at infinity of $\alpha_1$. The plane $\alpha_1$ also 
contains an affine line $L_O$ that corresponds to a conic $O$ on $C$ and which corresponds to the line at 
infinity of $\alpha$. It follows that $O$ extends $\alpha'$ to a projective plane, except possibly in the 
point that corresponds with the point at infinity of $L_O$ (meaning that we do not know whether the 
unique point of $O$ that does not correspond to any affine point of $L_O$ is incident with all lines 
of $\pi'$ that neighbor that point). But the same reasoning with now $C_2$ in the role of $C$ 
shows that $O$ really extends $\alpha'$ to a projective plane, and the claim is proved. 

Let us denote the point set of this projective plane by $P$.

Now, the conic $O$ lies in the $6$-space $\<C_1,\alpha_1\>$. The latter does not contain 
$y$. So, all points of $\alpha$ which are the projection of a point of $X$ that corresponds with an 
affine point of $\alpha_1$ are contained in the two $6$-spaces $\<C,\alpha\>$ and 
$\<C_1,\alpha_1\>$, which intersect in a $5$-space (since $y\in\<C,\alpha\>$; the intersection 
cannot 
have dimension $<5$ as $O$ spans a plane, hence a codimension 1 subspace of $\<C\>$). 
Consequently, we see that $P$ spans a $5$-space, and so it defines a Veronesean embedding of 
$\PG(2,\K)$. By \cite{Sch-Mal:**}, $P$ defines a quadric Veronese variety $\mathcal{V}$. 

By the properties of $\alpha'$ mentioned above, it is now also clear that $\chi$ defines a bijection 
between $P$ and the point set of $\pi^*_Y$. Also, as $y$ does not belong to $\<\mathcal{V}\>$, and 
$C$ can be regarded as an arbitrary conic of $\mathcal{V}$, we see that $\<\mathcal{V}\>$ and 
$\pi_Y$ are disjoint.  
\end{proof}

This proposition hints at the following construction of $X$. Consider the plane $\pi_Y$, and 
consider the quadric Veronese variety $\mathcal{V}$ in a complementary $5$-space. These objects are 
projectively unique. As we saw, there is a collineation, which we can denote by $\chi$, from 
$\mathcal{V}$ to $\pi^*_Y$ such that for each point $p\in\mathcal{V}$, the point set  
$\<p,\chi(p)\>\setminus\chi(p)$ is contained in $X$, and all points of $X$ arise in this way. In order to 
show that $X$ is projectively unique, there remains to nail down the collineation $\chi$. Since all 
linear collineations are projectively equivalent, it suffices to show that $\chi$  is linear, i.e., $\chi$ 
preserves the cross-ratio. This will be established in our last proposition.

\begin{cor}\label{corH3} Let $p_1,p_2,p_3,p_4$ be four points on a common conic of 
$\mathcal{V}$ corresponding to the singular planes $\pi_{L_1},\pi_{L_2},\pi_{L_3},\pi_{L_4}$, 
respectively. Then $L_1,L_2,L_3,L_4$ are concurrent lines of $\pi_Y$ and the cross-ratios 
$(p_1,p_2;p_3,p_4)$ and $(L_1,L_2;L_3,L_4)$ are equal.
\end{cor}
\begin{proof}
The fact that $L_1,L_2,L_3,L_4$ are concurrent lines of $\pi_Y$ follows from the fact that $\chi$ 
defines a collineation. Now let $y$ be the common point of the $L_i$, $i=1,2,3,4$. The points 
$p_1,p_2,p_3,p_4$ lie on singular lines together with $y$. If we consider the projection of all 
points of $X$ that lie on a singular line together with $y$, then, according to 
Proposition~\ref{propH1}, we obtain a normal rational cubic scroll, and $p_1,p_2,p_3,p_4$ can be 
seen as four points on the base conic, while the projections of $L_1,L_2,L_3,L_4$ are the four 
corresponding points on the vertex line. Just by the very definition of a normal rational cubic scroll, 
the assertion follows. 
\end{proof}

Now we treat the case $|\K|=2$. Let $x\in X$. Since $T_x$ contains at most $7$ planes sharing pairwise at most a line, we see that there are at most $7$ different tubes through $x$. This implies that $N\leq 4+7=11$, hence finite. Let $n=|X|\in\mathbb{N}$ be the number of points. Notice that not all tubes contain $x$ (indeed, consider two tubes $C_1,C_2$ through $x$, then no tube through the points $x_i\in C_i$, $i=1,2$, where $x_i$ is not collinear to $x$ on $C_i$, contains $x$). Let $C$ be a tube not through $x$. Since $C$ has three generators, there are at least three tubes connecting $x$ with $C$. So the number $n_x$ of tubes through $x$ satisfies $3\leq n_x\leq 7$, and it follows from the previous that, if $n_x=3$, then the three tubes through $x$ all meet in a fixed singular affine line. Let $x_0\in X$ be such that there are two tubes through $x_0$ only meeting in $x$ ($x_0$ exists by Lemma~\ref{lemmaH2}). Then $n_{x_0}\geq 4$. Let, for $x\in X$, $g_x$ be the number of generators through $x$, then one calculates that $|X|=4n_x+g_x+1$, where obviously $1\leq g_x\leq n_x$. We also have $g_{x_0}\geq 2$ by assumption. Hence $|X|\geq 19$. 

For $x\in X$, let $T_x'$ be the $3$-space obtained from $T_x$ by factoring out $x$. Then every tube $C$ through $x$ gives rise to a point-line flag in $T_x'$, where the point corresponds to the singular line of $C$ through $x$, and the line to the tangent plane to $C$ at $x$. It follows from (H2) that the set of all tubes through $x$ yields a set of such flags with the properties that, if two lines of different flags meet, then they must meet in their point (and therefore they have their point in common). It follows for instance that $g_x\leq 5$ since $T'_x$ contains $g_x$ disjoint lines. 

Now we consider the different possibilities for $g_x$. If $g_x=1$, then $n\in\{22,26,30\}$.  Note that, if $g_x=2$, then looking in $T_x'$, the only possibilities are $n_x\in\{2,3,4,5,6\}$, hence if $g_x=2$, then $n\in\{19,23,27\}$. If $g_x=3$, then likewise $n_x\in\{3,4,5,6\}$, hence $n\in\{20,24,28\}$. If $n_x=4$, then $g_x=4$ and $n=21$. Finally, if $g_x=5$, then $n_x=5$ and $n=26$. 

Since all these possibilities yield different possibilities for $n$, except for $n=26$, we see that either $n_x$ is constant, or there exist $x\in X$ with $(g_x,n_x)=(1,6)$ and $y\in X$ with $(g_y,n_y)=(5,5)$. 
Suppose first the latter happens, and let $x'$ be the other point on the unique singular line through $x$.

The five points $u$ on a common singular line with $y$ also have $g_u=5$. Let $t$ be the vertex of the tubes through $x$. Suppose there is some point $w\in X\setminus\{x,x'\}$ with $g_w=1$. Then all tubes through $w$ also have vertex $t$, and so the singular lines of these tubes coincide with the singular lines of the tubes through $x$. hence all points off the tube $X[x,w]$ are contained in singular lines that are contained in at least two tubes; hence $g_z=1$ for all such points $z$. There are now not enough points $u$ left with $g_u=5$. Hence $x,x'$ are unique with $g_x=1$. Let $y'$ be any point different from $x,x',y$, but contained in the tube $X[x,y]$. Let $C'$ be any tube through $y$ not containing $x$. Then there are unique tubes through $y'$ and the five respective points of $C'\setminus\{y\}$. Since these tubes can have at most one point in common with $C'$, these must all be different. But then one of these must coincide with $X[y,y']$, since there are only $5$ tubes through $y'$. That is clearly a contradiction.

We conclude that $g_x=g_y$ and $n_{x}=n_{y}$ for all $x,y\in X$. By Lemma~\ref{lemmaH2} we have $g_{x}>1$. Moreover
this implies that the number of cones $\frac{nn_{x}}{6}$ has to be an integer. This excludes $(g_{x},n_{x},n)\in \{(2,4,19);(2,5,23);(3,4,20);(5,5,26)\}$.

Consequently $(g_{x},n_{x},n)\in\{(2,6,27);(3,5,24);(3,6,28);(4,4,21)\}$. 

If $g_x=n_x=4$, then $\cG(X)$ is a linear space. For any point $x$, there is a line $L$, of size $6$, not incident with $x$. hence the number of lines through $x$ joining a point of $L$ is at least $6$, implying $n_x\geq 6$, a contradiction. This rules out the case $(4,4,21)$.


Now suppose $g_x=2;n_{x}=6,n=27$. Since $n_x>g_x$ in this case, there is some singular line $L$ contained in at least two tubes $C_1,C_2$. Let $L_1$ be a singular line of $C_1$ with $L_1\neq L$ and let $L_2$ be a singular line of $C_2$ with $L_2\neq L$. Suppose $L_i=\{t,x_i,y_i\}$, $i=1,2$. Suppose the tube $X[x_1,x_2]$ does not contain $L_1$. Then neither $X[y_1,y_2]$ contains $L_1$, and, since $X$ does not contain three points on a line,  the intersection $\<x_1,x_2\>\cap\<y_1,y_2\>$ is the common vertex of the tubes $X[x_1,x_2]$ and $X[y_1,y_2]$, and so $x_1$ and $x_2$ are on a singular line. Likewise $x_1$ and $y_2$ are on a singular line. This implies $g_{x_1}\geq 3$, a contradiction. Hence the singular lines through $t$ form a linear space with line sizes 3 defined by the tubes through $t$. Such a space has at least seven elements, and each singular line is contained in at least three tubes. Hence we have shown that, if a singular line is contained in at least two tubes, it is contained in at least three. In $T'_x$, this implies that we have two sets of point-line flags, each set with common point, and one set contains at least three members. Then it is obvious that this set contains at most four elements and the other set contains at most two elements, hence exactly one!
Each point $x$ is contained in a unique singular lines that is contained in at least two tubes, with common vertex $t$. The inverse of the application $x\mapsto t$ defines a partition of $X$, each partition set consisting of $2(2(n_x-1)+1)=4n_x-2$ points. Hence there are $\frac{4n_x+3}{4n_x-2}=\frac{19}{14}$ partition classes, which is not an integer, excluding the case $(2,6,27)$.

Hence $g_x=3$. Note that the structure of point-line flags in $T'_x$ implies that every singular line is contained in either one or two tubes. Let $L=\{t,x,y\}$ be a singular line, with $t\in Y$, and such that there are exactly two cylindric spaces through $L$. Let $L,L_1,M_1$ and $L,L_2,M_2$ be the singular lines of the corresponding tubes. Then $L_1$ cannot be contained in tubes together with $L_2$ and $M_2$, as in that case $L_1$ would be contained in at least three tubes. As in the previous paragraph, we deduce that, if $x_1\in L_1\cap X$, and $x_2\in L_2\cap X$, and $L_1\cup L_2$ is not contained in a cylindric space, then $\<x_1,x_2\>$ is a singular line. Since there are only three singular lines through $x_1$, we deduce that we may assume that $L_1,L_2$ are contained in a cylindric space, and so are $M_1,M_2$, but $L_1$ and $M_2$ are not contained in a cylindric space, and neither are $L_2,M_1$.   By symmetry, it follows that there is a unique additional singular line $L'$ through $t$ contained in the cylindric spaces determines by $L_1,L_2$ and by $M_1,M_2$.  Hence the set of singular lines through a fixed point of $Y$ covers 12 points of $X$. The intersection of two such 12-point sets contains at most 4 points (such as those of $L_1\cup M_2$). 
If $n=24$, then every point is contained in a unique singular line which is contained in a unique cylindric space. No singular lines of such unique cylindric space is contained in a second cylindric space, as otherwise we can apply our argument above and obtain a 12-set. So those cylindric spaces partition $X$, and there are exactly 4 of them. But every other tube intersects each of these cylindric spaces in at most one point, contradicting the fact that a tube contains 6 points of $X$. 

Hence $n=28$. Now with the above, it is not difficult to see that any 12-set as defined before forms the set of points neighboring a line in a level 2 Hjelmslev plane over the dual numbers of size 4.  

The result now follows as in the general case. 

The proof of Theorem~\ref{theo2} is now complete.


\section{Some 2-dimensional associative algebras}\label{algebras}

In this section we classify the algebras in the hypothesis of our Main Result.

\begin{theorem}\label{2algebras}
Let $\K$ be a commutative field and $V$ a $2$-dimensional vector space over $\K$ with scalar product denoted by $r\cdot v$. Let $\times$ be a commutative and associative multiplication on $V$ such that
\begin{itemize}
\item[\emph{(V1)}] $(r\cdot v+w)\times u=r\cdot (v\times u) + w\times u$, for all $r\in\K$ and all $u,v,w\in V$,
\item[\emph{(V2)}] each vector is the product of two other vectors. 
\end{itemize}
Then either $V,+,\times$ is a quadratic field extension, or $V,+,\times$ has the structure of the dual numbers, or $V,+,\times$ is the direct product $\K\times\K$. In each case there is a unique identity element $\mathbf{1}$ with respect to $\times$ and there exists a unique (linear) automorphism $\sigma$ of $V$ of order at most $2$ fixing $\mathbf{1}$ and such that both $v+v^\sigma$ and $v\times v^\sigma$ belong to $\K\cdot\mathbf{1}$, for all $v\in V$.
\end{theorem}

\begin{proof}
If there are no zero divisors, then the result is obvious (and $\sigma$ must be the unique nontrivial element in the Galois group if the extension is separable; if not then $\sigma$ must be the identity).

Now let the vector $w$ be a zero divisor and choose a vector $v$ linear independent from $w$. There are two possibilities: either $w\times w=0$ or we can rechoose $v$ such that $v\times w=0$. Note that not both possibilities can occur at the same time by condition (V2). Indeed, as otherwise the products of two vectors only generate a 1-dimensional vector space, namely the multiples of $v\times v$.

\textbf{First Case: $w\times w=0$.} 
Associativity on $v\times w\times w$ implies that $v\times w$ is a scalar multiple of $w$ (use (V2) too). We can rechoose $v$ such that $v \times w=w$. Associativity implies that $(v\times v)\times w=v\times w$, which yields $v\times v=v+tw$, for some $t\in\K$. Replacing $v$ by $v-tw$, we may assume that $t=0$. Hence we obtain the dual numbers. It is easy to see that $v$ is an identity element (the unique one of course) and that $\K\rightarrow \K v:k\mapsto k\cdot v$ is a ring monomorphism. Moreover, $\sigma$ must fix $v$ and map $w$ to $-w$.

\textbf{Second Case: $v\times w=0$.}
 Associativity on  $v\times v\times w$ implies that $v\times v$ is a nontrivial scalar multiple of $v$ (use (V2) too), say $v\times v=kv$, $k\in\K^\times$. Replacing $v$ by  $k^{-1}v$, we may assume that $k=1$. Similarly, we may assume that $w\times w=w$. Hence we obtain the direct product $\K\times \K$. A direct check reveals that the diagonal mapping $\K\rightarrow \K (v+w):k\mapsto k\cdot (v+w)$ is a ring monomorphism, and $v+w$ is an identity element for $\times$. It is also easily checked that $\sigma$ must act as $(rv+sw)^\sigma=sv+rw$. This is never the identity. 
\end{proof}

Now note that the three algebras obtained above are matrix algebras over $\K$. More exactly, if $V$ is a quadratic field 
extension, say with respect to the quadratic polynomial $x^2-tx+n$, then $V$ can be thought of to consist of the matrices 
$$\left[\begin{array}{cc}x & y \\ -ny & x+ty\end{array}\right], x,y \in\K.$$ 
If $V$ is the direct product of $\K$ with itself, then obviously we can identify $V$ with the matrix algebra of the diagonal 
matrices. However, in order to allow for a uniform treatment, we identify the algebra with the algebra of matrices  $$\left[\begin{array}{cc}x & y \\ 0 & x+y\end{array}\right], x,y \in\K.$$
In other words $(1,0)$ identifies with 
$\left[\begin{array}{cc}1 & -1 \\ 0 & 0\end{array}\right], x,y \in\K$
and $(0,1)$ identifies with
$\left[\begin{array}{cc}0 & 1 \\ 0 & 1\end{array}\right], x,y \in\K$.

If $V$ is the ring of dual numbers of $\K$, then it can be thought of to consist of the matrices  $$\left[\begin{array}{cc}x & y \\ 0 
& x\end{array}\right], x,y \in\K.$$ In each case, the automorphism $\sigma$ is given by assigning to each matrix $A$ its adjugate 
$A^{\ad}$ (which is the transposed of the matrix of cofactors and it equals the determinant times the inverse, if the latter 
exists). More exactly, $$\sigma:V\rightarrow V: \left[\begin{array}{cc}a & b \\ c & d\end{array}\right]\mapsto  
\left[\begin{array}{cc}d & -b \\ -c & a\end{array}\right],$$ for all appropriate $a,b,c,d$ such that the given matrix belongs to $V$. 

In order to allow for a uniform treatment, we note that the second and the third of the three above algebras can be described in 
exactly the same way as the first one by putting $(t,n)$ equal to $(1,0)$ and $(0,0)$, respectively.  In fact, isomorphic 
descriptions are obtained if one now varies $(t,n)$ over all possible values in $\K\times\K$. If the quadratic polynomial 
$x^2-tx+n$ has no solutions, exactly one solution, or two distinct solutions in $\K$, then we we obtain a matrix algebra isomorphic 
to first, second or third case above, respectively.

So from now on, we assume that we have been given the pair $(n,t)\in\K\times\K$ arbitrarily, and we consider the corresponding 
matrix algebra $V$. But we emphasize that there are other matrix representations of these algebras not captured by assigning values 
to $(t,n)$. Also for these representations, the results below hold, with a similar proof. 

We write each element of $V$ as $x\mathfrak{R}+y\mathfrak{I}$, with $$\mathfrak{R}= \left[\begin{array}{cc}1 & 0 \\ 0 & 1\end{array}\right]  \mbox{ and } \mathfrak{I}= \left[\begin{array}{cc}0 & 1 \\ -n & t\end{array}\right].$$

Although $\mathfrak{R}$ is the identity matrix, we insist on keeping the notation $\mathfrak{R}$ for it, as $\mathfrak{R}$ could be a different matrix for another representation of the algebra $V$. And all results below hold for arbitrary $2\times2$ matrix representations. 

For an arbitrary element $A$ of $V$, write $x(A)$ and $y(A)$ for the unique elements of $\K$ such that $A=x(A)\mathfrak{R}+y(A)\mathfrak{I}$. This is a canonical way to write $V$ as a $2$-dimensional vector space over $\K$. 

The following lemma is easily verified by direct computation.

\begin{lemma} \label{frak}
For every matrix $M\in V$, we have 
$M=\left[\begin{array}{cc}x(\mathfrak{R}M)&y(\mathfrak{R}M)\\ x(\mathfrak{I}M)&y(\mathfrak{I}M)\end{array}\right]$.
\qed
\end{lemma}

And also the next lemma is straightforward and requires no explicit proof.

\begin{lemma}\label{frak2}
Let $M,N\in V$, then the six determinants corresponding to the six $2\times 2$ matrices obtained by deleting two columns in the 
$2\times 4$ matrix constructed from $M$ and $N$ by juxtaposition, equal, up to sign and as a multiset, the four entries of the matrix 
$N^{\ad}M$ and the two determinants $\det M$ and $\det N$. \qed
\end{lemma}

\begin{rem}\em We cannot dispense with the commutativity in Theorem~\ref{2algebras} as the matrix algebra of $2\times 2$ matrices with zero 
second row is a counterexample. Although it is also a matrix algebra 
(as all associative 2-dimensional algebras), it does not suit our 
purposes here.
\end{rem}    


\section{Equivalence of $V$-sets}

In this section, we consider the constructions of the various $V$-sets given in Subsection~\ref{planes}, and show some equivalence amongst them. 

We start with the $V$-sets defined by reduction and the ones defined by juxtaposition. 

\begin{prop} \label{theoV1}
Let $V$ be a commutative and associative $2$-dimensional algebra over some commutative field $\K$ such that each vector is the product of two other vectors. Then the $V$-set defined by reduction is projectively equivalent with the $V$-set defined by juxtaposition.
\end{prop} 

\begin{proof}
As we learned in the previous section, we can write any element of $V$ as $x\mathfrak{R}+y\mathfrak{I}$, and we will write this vector as $(x,y)$. Let $R_1(x,y)$ be the first row of this matrix, and $R_2(x,y)$ be the second (both are $2$-dimensional vectors over $\K$). Take any point of $\mathcal{G}(V)$, i.e., a triple $T=((x,y),(x',y'),(x'',y''))$ of elements of $V$, up to a scalar factor in $V^*$, and such that no common  multiple of $(x,y)$, $(x',y')$ and $(x'',y'')$ using a nonzero factor in $V$ is zero.  The $1$-space over $V$ defined by $T$ is given by all multiples over $V$ of $T$, and viewed as vector space over $\K$, the $2$-space over $\K$ defined by $T$ is hence generated by $\mathfrak{R}T$ and $\mathfrak{I}T$. Lemma~\ref{frak} implies that this space is hence generated by the two points obtained by juxtaposition of the matrices of $(x,y), (x',y')$ and $(x'',y'')$ and taking the points with coordinates in $\K$ given by the respective rows in this juxtaposition. This completes the proof of the proposition.
\end{proof}

We now look at the $V$-sets defined by matrices.

\begin{prop}\label{theoV2}
Let $V$ be a commutative and associative $2$-dimensional algebra over some commutative field $\K$ such that each vector is the product of two 
other vectors. Then the $V$-set defined by matrices is projectively equivalent with the $V$-set defined by juxtaposition.
\end{prop}

\begin{proof}
The coordinates in $\PG(8,\K)$ of the point corresponding with the point $V^*(M,N,L)$, where $M,N,L$ are elements of the matrix algebra 
$V$, are given by the nine values $$\det M,\det N,\det L,x(M^{\ad}N),y(M^{\ad}N), x(N^{\ad}L),y(N^{\ad}L), x(L^{\ad}M),y(L^{\ad}M).$$
Now Lemma~\ref{frak2} implies that the line Grassmannian coordinates in $\PG(14,\K)$ of the line of $\PG(5,\K)$ determined by the points whose 
coordinates correspond with the two rows of the $2\times 6$ matrix obtained by juxtaposing $M,N,L$, are given by the three determinants 
$\det M,\det N,\det L$, and by the twelve entries of the matrices $M^{\ad}N,N^{\ad}L$ and $L^{\ad}M$. But the entries in the first rows of these coincide with $x(M^{\ad}N)$, $y(M^{\ad}N)$, $x(N^{\ad}L)$, $y(N^{\ad}L$, $x(L^{\ad}M)$ and $y(L^{\ad}M)$, whereas the entries in the second rows
are simple linear combinations of the other entries. This determines a projective linear transformation between these two $V$-sets.   
\end{proof}

Finally, we consider $V$-sets defined by parametrization. One easily sees that the construction of these $V$-sets is independent of the chosen element $\zeta$ (as any other element is a linear $\K$-combination of $1$ and $\zeta$). 

\begin{prop}\label{theoV3}
Let $V$ be a commutative and associative $2$-dimensional algebra over some commutative field $\K$ such that each vector is the product of two 
other vectors. Then the $V$-set defined by matrices is projectively equivalent with the $V$-set defined by juxtaposition if and only if $V$ is neither isomorphic to the dual numbers, nor isomorphic to an inseparable quadratic extension of $\K$. If $V$ is isomorphic to the dual numbers over a field of characteristic distinct from $2$, then the $V$-set defined by parametrization is projectively equivalent with a quadric Veronese variety over $\K$. If $V$ is isomorphic to the dual numbers over a field of characteristic $2$, or $V$ is isomorphic to an inseparable quadratic extension of $\K$, then the $V$-set defined by parametrization is projectively equivalent with a projective plane over $\K^2$, the field of squares of $\K$, or over $V^2$, the field of squares of $V$, respectively.
\end{prop} 

\begin{proof}
We may put $\zeta=\mathfrak{I}$. Recall that we have fixed the number $t$ and $n$ in $\K$.  Then one easily calculates for two arbitrary matrices $M,L\in V$ that \begin{eqnarray*} M^{\ad}L+L^{\ad}M & = & 2x(M^{\ad}L) + ty(M^{\ad}L),\\  \mathfrak{I}M^{\ad}L+\mathfrak{I}^{\ad}L^{\ad}M & = & tx(M^{\ad}L) + (t^2-2n)y(M^{\ad}L).\end{eqnarray*} These are again linear combinations of $x=x(M^{\ad}L)$ and $y=y(M^{\ad}L)$, and they are invertible (meaning that we can calculate $x$ and $y$ from the two right hand sides) if and only if
$$\det\left[\begin{array}{cc}2 & t \\ t & t^2-2n\end{array}\right]\neq 0.$$ This happens if and only if $t^2-4n\neq 0$. If the characteristic of $\K$ is unequal $2$, then this happens if and only if $V$ is not isomorphic to the dual numbers (since then, the equation $\alpha^2+t\alpha+n=0$ has either $0$ or $2$ different solutions in $\K$). If the characteristic of $\K$ equals $2$, then this happens if and only if $t\neq 0$. If in characteristic $2$ we have $t=0$, then the equation $\alpha^2+n=0$ either has one or no solutions and we either have the dual numbers or an inseparable quadratic extension of $\K$. If in characteristic $2$ we have $t\neq 0$, then the equation $\alpha^2+t\alpha+n=0$ has either no solutions (and then $V$ is a separable quadratic extension of $\K$) or exactly $2$ solutions (since, if $\beta$ is a solution, also $\beta+t$ is one). In the latter case $V$ is the direct product $\K\times\K$.

Suppose now that the characteristic of $\K$ is unequal $2$ and $V$ is isomorphic to the dual numbers over $\K$. Then, for ease of 
computations, we may put $t=n=0$. We easily obtain that the point $V^*(M,N,L)$ is mapped under the Veronese  correspondence to the 
point $(x(M)^2,x(N)^2,x(L)^2,2x(M)x(N),0,2x(N)x(L),0,2x(L)x(M),0)$, which clearly defines a quadric Veronese variety. 

If the characteristic of $\K$ is equal to $2$ and $t=0$, then the point $V^*(M,N,L)$ is mapped under the Veronese  correspondence to the 
point $(x(M)^2+ny(M)^2,x(N)^2+ny(N)^2,x(L)^2+ny(L)^2,0,0,0,0,0,0)$, which clearly defines a projective plane over the field of squares of $\K$, 
if $n$ is itself a square in $\K$, and over the field of squares of $V$, if $n$ is not a square in $\K$, but then it is a square in $V$.  
\end{proof}

For completeness' sake, we now show the following proposition.  

\begin{prop} Let $V$ be a commutative and associative $2$-dimensional algebra over some commutative field $\K$ such that each vector is the product of two 
other vectors. The set $S_1$ of rank $1$ Hermitian $3\times 3$ matrices over $V$ coincides with the set $S_2$ of all scalar multiples of matrices $(M~N~L)^t(M~N~L)^{\ad}$, for $V^*(M,N,L)$ a point of $\cG(V)$. \end{prop}

\begin{proof} Clearly, $S_2\subseteq S_1$. Now let $$H=\left[\begin{array}{ccc} k_1 & R_3 & R_2^{\ad}\\ R_3^{\ad} & k_2 & R_1 \\ R_2 & R_1^{\ad} & 
k_3\end{array}\right],$$ with $k_1,k_2,k_3\in\K$ and $R_1,R_2,R_3\in V$ be a Hermitian matrix of rank 1. Suppose first 
$(k_1,k_2,k_3)\neq(0,0,0)$. Then we may without loss assume $k_1=1$. But then $H=(M~N~L)^t(M~N~L)^{\ad}$, with $(M,N,L)=(1,R_3^{\ad},R_2)$ 
because the first row and the first column of left hand and right hand sides coincide, and we claim that this determines the rest by the fact that the ranks 
of both matrices are $1$. 

Indeed, suppose first that $R_3=0$. Then $R_3^{\ad}R_3=0$ and we must show $k_2=0$. If $k_2\neq0$, then $k_2$ is 
invertible, and so $a$ times the first row plus $b$ times the second row (with $(a,b)\neq(0,0)$) always gives a nonzero row as only for $a=b=0$ 
the first two entries can be $0$. Likewise $R_2^{\ad}R_3=0$ and we must show $R_1=0$. Suppose $a$ times the first row plus $b$ times the second 
row is zero, with $a,b\in L$ and such that no $c\in L\setminus\{0\}$ exists with $ac=bc=0$. Then looking at the first entry, we see $a=0$. Hence 
$bR_1=0$, contradicting the nonexistence of $c$ with the just mentioned properties. Now suppose $R_3\neq 0$. Let $a,b\in L$ be as before. Then 
$a+bR_3^{\ad}=0=aR_3+bk_2$. Remember $k_2\in \K$ is invertible. It follows that $b=-aR_3k_2^{-1}$ and so $a-aR_3R_3^{\ad}k_2^{-1}=0$. If the 
second term of the left hand side of this equality is $0$, then $a=0$ and this leads to $b$ invertible, contradicting $bR_3^{\ad}=0$. Hence $R_3$ 
is invertible, and so $b=-aR_3^{-\ad}$. This implies that both $a$ and $b$ are invertible, and so $k_2=-aR_3b^{_1}=R_3R_3^{\ad}$. It remains to 
determine $R_1$ when both $R_2$ and $R_3$ are nonzero. Let again $a,b$ be as before. Then $a+bR_3^{\ad}=0=aR_2^{\ad}+bR_1$. We calculate 
$a=-bR_3^{\ad}$ and so $bR_1=bR_2^{\ad}R_3^{\ad}$.  Remember that, by the foregoing, we now know that both $R_3$ and $R_2$ are invertible. 
Hence, similarly as before,  this implies that both $a$ and $b$ are invertible. Consequently, $R_1=R_2^{\ad}R_3^{\ad}$, which completes the 
proof of our claim. 

Now assume $k_1=k_2=k_3=0$. We may then assume $R_3\neq 0$. Let $a,b$ be as before. Then $0=aR_3=bR_3^{\ad}=aR_2^{\ad}+bR_1$. So $R_3$ cannot be 
invertible. Neither $a$ nor $b$ are invertible, and so $V$ is not isomorphic to the dual numbers, as this would imply that $a\in b\K$ and so $a^2=ab=0$, a contradiction. So $V$ is $\K\times\K$. We can identify $V$ with $\K\times\K$ and write every element of $V$ as a pair $(x,y)$ with coordinatewise addition and multiplication. By rescaling we may assume without loss that $R_3=(1,0)$. Hence $a=(0,A)$ and $b=(B,0)$, for some $A,B\in\K^\times$. It follows from $aR_2^{\ad}+bR_1=0$ that $R_2=(0,r_2)$ and $R_1=(0,r_1)$, for some $r_1,r_2\in\K$.

If $r_1\neq 0$, then linear dependence of the second and the third row implies easily that $r_2=0$. In this case we see that $(M,N,L)=(R_3,R_3^{\ad},R_1^{\ad})$. If $r_1=0$, then we immediately have $(M,N,L)=(R_3,R_3^{\ad},R_2)$, and the proof is complete.
\end{proof}

Finally, we show that all $V$-sets we considered span an $8$-dimensional space. By isomorphism, it is enough to do so for the $V$-sets defined by matrices.  

\begin{prop}\label{span8}Let $V$ be a commutative and associative $2$-dimensional algebra over some commutative field $\K$ such that each vector is the product of two 
other vectors. Then the $V$-set defined by matrices spans an $8$-dimensional projective space.
\end{prop}

\begin{proof}
We consider the nine Hermitian matrices $(M~N~L)^t(M~N~L)^{\ad}$ with $(M,N,L)$ equal to
$$\begin{array}{lll}
(\mathfrak{R},0,0), \hspace{1cm}& (0,\mathfrak{R},0),\hspace{1cm} & (0,0,\mathfrak{R}),\\
(0,\mathfrak{R},\mathfrak{R}),&(\mathfrak{R},0,\mathfrak{R}),& (\mathfrak{R},\mathfrak{R},0),\\
(0,\mathfrak{R},\mathfrak{I}),&(\mathfrak{I},0,\mathfrak{R}),& (\mathfrak{R},\mathfrak{I},0),\end{array}$$ respectively. Then it is easily checked that these are $\K$-linearly independent.
\end{proof}


\section{Some properties of $\cG(V)$}

In this section we show some (transitivity and neighboring) properties of the geometry $\cG(V)$, for $V$ a commutative and associative $2$-dimensional algebra over some commutative field $\K$ such that each vector is the product of two 
other vectors. These properties will allow us later to prove that each of the constructions of $V$-sets $X$ from $V$ turn $X$ into a $\cC$-Veronesean set when endowed with the subspaces determined by the lines of $\cG(V)$ by considering the span of their points.

The results in this section are well known in case $V$ is a field or the dual numbers over a field. For $V$ the direct product of two copies of the same field, it might be harder to find the appropriate literature. In any case, the properties we mention are similar to the ones proved in \cite{SV} for the Hjelmslev-Moufang planes. In fact, we could refer to these proofs, if the latter did not exclude characteristics 2 and 3. Hence, for completeness sake, we include full proofs. 

We start with defining a neigboring relation and proving some properties. Then we will use these properties to state and prove a rather strong transitivity property. 

\subsection{Neighboring properties of $\cG(V)$}

Throughout this section, let $r$ and $s$ in $V$ be such that $V_0=\K r\cup\K s$ (remember that $V_0=V\setminus V^*$). For $V$ a field, $s=r=0$, for $V$ the dual numbers over $\K$, we may assume $r=s\neq 0$, and for $V$ the direct product of $\K$ with itself, we have $\K r\neq \K s$. Note that $r$ and $s$ play the same role; they can be interchanged in any argument, and we will not explicitly state this every time. Note also that $rs=0$, and, in case $r\neq 0$, the set of elements $t\in V$ such that $tr=0$ is precisely $\K s$.

\begin{defi} \label{defnei} \rm Two points $V^*(x,y,z)$ and $V^*(u,v,w)$ of $\cG(V)$ are called \emph{neighboring}, 
denoted $V^*(x,y,z)\sim V^*(u,v,w)$, if there exist elements $k,\ell\in V$ such that $k(x,y,z)=\ell(u,v,w)$. Similarly for lines. 
Also, a point $V^*(x,y,z)$ and a line $V^*[a,b,c]$ are \emph{neighboring}, denoted $V^*(x,y,z)\sim V^*[a,b,c]$, if $ax+by+cz\in V_0$.
\end{defi}

We note that, if two distinct points $V^*(x,y,z)$ and $V^*(u,v,w)$ of $\cG(V)$ are neighboring, and $k,\ell\in V$ are such that $k(x,y,z)=\ell(u,v,w)$, then both $k$ and $\ell$ are zero divisors and $k\in \K \ell$. Indeed, if $k$ is invertible, then clearly the points are equal; if $k$ and $\ell$ would be linearly independent over $\K$, then we know $k^2\in\K k$ and $k\ell=0$, hence $k^2(x,y,z)=(0,0,0)$ without $k^2$ being $0$, a contradiction. 

Our first concern is to find equivalent geometric conditions for the neighboring relation. To that aim, we consider two lines $L_1=:V^*[a_1,b_1,c_1]$ and $L_2=:V^*[a_2,b_2,c_2]$ and we set 
$$A=: \left|\begin{array}{cc} b_1 & c_1\\ b_2 & c_2\end{array}\right|, \vspace{1cm} B=:\left|\begin{array}{cc} c_1 & a_1\\ c_2 & a_2\end{array}\right|, \vspace{1cm} C=:\left|\begin{array}{cc} a_1 & b_1\\ a_2 & b_2\end{array}\right|.$$
\begin{lemma}\label{N1}
With the above notation, the lines $L_1$ and $L_2$ are not neighboring if and only if $V^*(A,B,C)$ is a point of $\cG(V)$.
\end{lemma}

\begin{proof}
First suppose $L_1$ and $L_2$ are neighboring. Then, without loss, we may assume $r(a_1,b_1,c_1)=r(a_2,b_2,c_2)$ (after multiplying one of the triples with a unit of $V$). Then $$rA=r\left|\begin{array}{cc} b_1 & c_1\\ b_2 & c_2\end{array}\right|=\left|\begin{array}{cc} rb_1 & rc_1\\ b_2 & c_2\end{array}\right|=\left|\begin{array}{cc} rb_2 & rc_2\\ b_2 & c_2\end{array}\right|=0.$$

Similarly $rB=rC=0$, hence $V^*(A,B,C)$ is not a point of $\cG(V)$.

Now suppose $r(A,B,C)=(0,0,0)$. Then $A,B,C$ are scalar multiples of $s$. This immediately implies that, if $a_1$ is a scalar multiple of $s$, and $a_2$ is not, then both $b_1$ and $c_1$ are (since $B$ and $C$ are), a contradiction to $L_1$ being a line. Hence, $a_1$ is a scalar multiple of $s$ if and only if $a_2$ is. If two of $a_1,b_1,c_1$ are scalar multiples of $s$, then this implies that $r(a_1,b_1,c_1)$ is a scalar multiple of $r(a_2,b_2,c_2)$, and so $L_1$ and $L_2$ are neighboring. Suppose now neither $a_1$ nor $b_1$ is a multiple of $s$. Then there exist nonzero scalars $k_a,k_b\in\K$ such that $ra_1=k_ara_2$ and $rb_1=k_brb_2$. But then one easily calculates $0=rC=r(k_a-k_b)a_2b_2$. Since none of $a_2,b_2$ is a multiple of $s$, we have $k_a=k_b$. Hence $r(a_1,b_1)=k_ar(a_2,b_2)$. If $c_1$ is not a multiple of $s$, then similarly, we obtain $r(a_1,b_1,c_1)=k_ar(a_2,b_2,c_2)$, and so $L_1$ and $L_2$ are neighboring. If $c_1$, and hence also $c_2$, is a multiple of $s$, then $rc_1=k_arc_2=0$, and hence $L_1$ again neighbors $L_2$.  
\end{proof}

\begin{lemma}\label{N2}
Suppose that $L_1$ and $L_2$ are not neighboring, and let $A',B',C'\in V$ be such that $$\left\{\begin{array}{rcl}A'a_1+B'b_1+C'c_1&=&0,\\
A'a_2+B'b_2+C'c_2&=&0.\end{array}\right.$$
Then $(A',B',C')$ is a multiple in $V$ of $(A,B,C)$. Conversely, every multiple $(A',B',C')$ of $(A,B,C)$ satisfies the above equations. 
\end{lemma}

\begin{proof}
By Lemma~\ref{N1}, $V^*(A,B,C)$ is a point and an elementary computation shows that it lies on both $L_1$ and $L_2$. Now 
suppose first that one of $A,B,C$ is invertible, say $A\in V^*$. Then the pair $(AB'-A'B,AC'-A'C)$ satisfies

$$\left\{\begin{array}{rcl}(AB'-A'B)b_1+(AC'-A'C)c_1&=&0,\\
(AB'-A'B)b_2+(AC'-A'C)c_2&=&0.\end{array}\right.$$

Since $A=b_1c_2-b_2c_1$ is invertible, we see, by multiplying the first equation with $c_2$ and subtracting $c_1$ times the second 
equation, that $AB'-A'B=0$, and similarly $AC'-A'C=0$. Hence $(A',B',C')=A^{-1}A'(A,B,C)$. 

If none of $A,B,C$ is invertible, then two of these are not scalar multiples of one another, say $A$ and $B$. Then $A+B$ is 
invertible, and we apply the first part of the proof to the lines $V^*(a_1,b_1-a_1,c_1)$ and $V^*(a_2,b_2-a_2,c_2)$ (one indeed 
checks easily that these are lines, and that they are not neighboring). The lemma now follows. 
\end{proof}

\begin{lemma}\label{N3}
Suppose that $L_1$ and $L_2$ are not neighboring. Then $L_1$ and $L_2$ have a unique point in common, namely $V^*(A,B,C)$. Also, if a point $V^*(x,y,z)$ of $L_1$ neighbors $L_2$, then it neighbors the intersection point $V^*(A,B,C)$ of $L_1$ with $L_2$. 
\end{lemma}

\begin{proof}
The first assertion follows directly from Lemma~\ref{N2}. As for the second assertion, we may suppose that $a_2x+b_2y+c_2z\in\K r$, hence, by Lemma~\ref{N2} again, the triple $(sx,sy,sz)$ is a multiple of $(A,B,C)$ and hence, by definition, neighboring $V^*(A,B,C)$, which is, by the first assertion, the intersection point of $L_1$ and $L_2$.
\end{proof}

\begin{lemma}\label{N4}
No point $P$ neighbors all points of a line $L$.
\end{lemma}

\begin{proof}
First suppose that one of the coordinates of $P$, say the first one, is invertible. It is easy to see that $L$ contains a point $P'$ whose first coordinate is $0$. Clearly, $P$ and $P'$ are not neighboring. 

Now suppose that no coordinate of $P$ is invertible. Then $\K r\neq\K s$ and without loss of generality, we may assume that $P$ is the point $V^*(k_x r, k_y s, k_z s)$, with $k_x,k_y,k_z\in\K$. As before, there is a point $Q:=V^*(0,y,z)$ of $L$, and we may assume $y\neq 0$. Then there is a different point $Q'=V^*(x',0,z')$ of $L$. If $Q\sim P$, then $(k_ys^2,k_zs^2)$ is a scalar multiple of $(ys,zs)$. If $Q'\sim P$, then $z'r=0$. Let $Q''$ be the point $V^*(x',y,z+z')$, which is on $L$, and suppose that $Q\sim P$, $Q'\sim P$ and $Q''\sim P$. If $r(k_x r,k_y s, k_zs)\in\K r(x',y,z+z')$, then $0=ry=r(z+z')=rz$, contradicting the fact that $Q$ is a point; if $s(k_x r,k_y s, k_zs)\in\K s(x',y,z+z')$, then $sx'=0$ and $(k_ys^2,k_zs^2)$ is a scalar multiple of $(ys,zs+z's)$, hence $z's=0$ and this, together with $x's=0$, contradicts the fact that $Q'$ is a point. 

The lemma is proved. 
\end{proof}

We can now show a geometric interpretation of the neighboring relation between a point and a line. 

\begin{lemma}\label{N5}
A point $P:=V^*(x,y,z)$ neighbors a line $L:=V^*(a,b,c)$ if and only if $P$ neighbors some point of $L$ if and only if $L$ neighbors some line through $P$. 
\end{lemma}

\begin{proof}
Suppose first that $P$ neighbors a point $p':=V^*(x',y',z')$ of $L$. We may assume that $r(x,y,z)=kr(x',y',z')$, for some $k\in\K$. Then $r(ax+by+cz)=a(rx)+b(ry)+c(rz)=kr(ax'+by'+cz')=0$ and so $ax+by+cz\in\K s\subseteq V\setminus V^*$. Hence $P$ neighbors $L$.

Now assume that $P$ neighbors $L$. The dual of Lemma~\ref{N4} implies that some line $M$ through $P$ is not neighboring $L$. Then $P$ neighbors the intersection point of $L$ and $M$ by Lemma~\ref{N3}.

The second ``if and only if'' follows by duality.
\end{proof}

We now show a geometric interpretation of the neighboring relation between points. 

\begin{lemma}\label{N6}
Two lines are neighboring if and only if they intersect in at least two common points. 
\end{lemma}

\begin{proof}
In view of Lemma~\ref{N3}, it suffices to show that two neighboring lines intersect in at least two common points. Suppose the lines $L_i=V^*[a_i,b_i,c_i]$ are neighboring. By possibly multiplying with a nonzero scalar, we may assume that $r(a_1,b_1,c_1)=r(a_2,b_2,c_2)$, i.e., $r(a_1-a_2,b_1-b_2,c_1-c_2)=(0,0,0)$. It follows that we can write $(a_1-a_2,b_1-b_2,c_1-c_2)=s(a_3,b_3,c_3)$, with $a_3,b_3,c_3\in(\K+r\K)^3\setminus\{(0,0,0)\}$ (and the part in $r\K$ of each coordinate can be chosen freely). This implies that $(a_3,b_3,c_3)$ is not a $V^*$-multiple of $(a_1,b_1,c_1)$, as this would imply that $(a_2,b_2,c_2)$ is a $V$-multiple of $(a_1,b_1,c_1)$, a contradiction. Hence $V^*[a_1,b_1,c_1]$ and $V^*[a_3,b_3,c_3]$ are two distinct lines, and we claim that we can choose $a_3,b_3,c_3$ such that they are not neighboring. Indeed, if $s(a_3,b_3,c_3)=ks(a_1,b_1,c_1)$, $k\in\K$, then $(a_2,b_2,c_2)=(1-ks)(a_1,b_1,c_1)$, a contradiction as this would imply that $L_1=L_2$; now assume that $s\neq r$ and for every choice of $a_3,b_3,c_3$, there exists $k\in \K$ such that $kr(a_1,b_1,c_1)=r(a_3,b_3,c_3)$. Since $r^2$ is a scalar multiple of $r$, this is a contradiction. So we may pick $a_3,b_3,c_3$ such that $V^*[a_1,b_1,c_1]$ and $V^*[a_3,b_3,c_3]$ are two distinct non-neighboring lines. The unique intersection point $P=V^*(x,y,z)$ is also incident with $V^*[a_2,b_2,c_2]$. So we already found one point of intersection. Now if $V^*(x_1,y_1,z_1)$ is any point of $L_1$, then $P_1=V^*(x+rx_1,y+ry_1,z+rz_1)$ is clearly a point on both $L_1,L_2$, and it is easy to see that we can choose $x_1,y_1,z_1$ such that $P_1\neq P$.  

The lemma is proved.
\end{proof}

\begin{defi}\rm A triple of points $(P_1,P_2,P_3)$ is called a \emph{proper ordered triangle} if $P_1,P_2,P_3$ are 
pairwise non-neighboring and the unique lines determined by $P_1,P_2$, by $P_2,P_3$ and by 
$P_3,P_1$, are pairwise non-neighboring.
\end{defi}

Note that this is a self-dual definition in the sense that, if $(P_1,P_2,P_3)$ is a proper ordered triangle, then so is any triple formed by the unique lines determined by these three points. 

We now establish a seemingly weaker condition for a a proper ordered triangle.

\begin{lemma}\label{triangle}
A triple of points $(P_1,P_2,P_3)$ is a proper ordered triangle if and only if we can re-order the points such that $P_1$ and $P_2$ are non-neighboring, and $P_3$ does not neighbor the line $P_1P_2$ through $P_1$ and $P_2$.
\end{lemma}

\begin{proof}
Since $P_3$ does not neighbor $P_1P_2$, is neither neighbors $P_1$ nor $P_2$ by Lemma~\ref{N5}. By the same reference, the line $P_1P_2$ neither neighbors the line $P_1P_3$ nor the line $P_2P_3$ (with similar self-explaining notation). If suffices to show that $P_2P_3\sim P_1P_3$ leads to a contradiction. Indeed, Lemma~\ref{N5} implies $P_1\sim P_2P_3$. Since $P_2P_3$ does not neighbor $P_1P_2$, the second assertion of Lemma~\ref{N3} implies $P_1\sim P_2$, our wanted contradiction.
\end{proof}

\begin{defi}
A \emph{proper ordered quadrangle} $(P_1,P_2,P_3,P_4)$ of points is a quadruple for which every ordered subtriple is a proper ordered triangle.
\end{defi}

Lemma~\ref{triangle} implies that a quadruple $(P_1,P_2,P_3,P_4)$ of points is a proper ordered quadrangle if and only if for some permutation $(i,j,k,\ell)$ of $(1,2,3,4)$ we have that $P_i$ does not neighbor $P_j$, $P_k$ does not neighbor $P_iP_j$, and $P_\ell$ does neither neighbor $P_iP_j$, nor $P_iP_k$, nor $P_jP_k$.

We now prove an algebraic characterization of proper ordered quadrangles. First a lemma.

\begin{lemma}\label{N7}
If $P_1,P_2$ are two non-neighboring points of $\cG(V)$, and $P_i=V^*(x_i,y_i,z_i)$, $i=1,2$, then every point of the line $P_1P_2$ can be written as $V^*(a_1x_1+a_2x_2,a_1y_1+a_2y_2,a_1z_1+a_2z_2)$, with $a_1,a_2\in V$, and at most one of $a_1,a_2$ is a scalar multiple of $r$. Similarly, at most one of $a_1,a_2$ is a scalar multiple of $s$. Conversely, for every $a_1,a_2\in V$, where at most one of $a_1,a_2$ is a scalar multiple of $r$ or of $s$, the expression $V^*(a_1x_1+a_2x_2,a_1y_1+a_2y_2,a_1z_1+a_2z_2)$ defines a point of the line $P_1P_2$. 
\end{lemma}

\begin{proof}
Put $\vec{x}_i=(x_i,y_i,z_i)$. Choose $a_1,a_2\in V$ such that at most one of $a_1,a_2$ is a scalar multiple of $r$ or of $s$. Then at most one of $ra_1\vec{x}_1$ and $ra_2\vec{x}_2$ is $(0,0,0)$. Also, $r(a_1\vec{x}_1+a_2\vec{x}_2)$ is nonzero as otherwise $P_1$ and $P_2$ are neighboring. Hence we have shown the last assertion.

Now let $P$ be any point of the line $P_1P_2$. Then the dual of Lemma~\ref{N4} asserts that we can find a line $L$ through $P$ not neighboring $P_1P_2$, and hence intersecting $P_1P_2$ in only $P$. Let $L:=V^*[a,b,c]$. Put $b_i=ax_i+by_i+cz_i$, $i=1,2$. We seek $a_1,a_2\in V$, not both a scalar multiple of $r$ or of $s$, such that $a_1b_1+a_2b_2=0$. If one of $b_1,b_2$ is $0$, then this is trivial, as well as if one of $b_1$ or $b_2$ is invertible. If both are scalar multiples of $r$, then we find scalars $a_1,a_2$; similarly if both are scalar multiples of $s$. Finally, if one is a scalar multiple of $r$, say $b_1$, and the other, say $b_2$, of $s$, then we can take $(a_1,a_2)=(s,r)$. In all cases we see that there is a point of $P_1P_2$ in $L$ which can be written as $V^*(a_1\vec{x}_1+a_2\vec{x}_2)$. This must necessarily be $P$, and so $P$ can be written in the desired form. 
\end{proof}

\begin{prop}\label{prop1}
Let $P_i:=V^*(x_i,y_i,z_i)$, $i=1,2,3,4$, be four points of $\cG(V)$. Then $(P_1,P_2,P_3,P_4)$ is a proper ordered quadrangle if and only if $$D:=\left|\begin{array}{ccc} x_1 & y_1 & z_1\\ x_2 & y_2 & z_2 \\ x_3 & y_3 & z_3\end{array}\right|\in V^*$$ and there exist $a_1,a_2,a_3\in V^*$ such that $(x_4,y_4,z_4)= \sum_{i=1}^3a_i(x_i,y_i,z_i)$.
\end{prop}
 
\begin{proof} Throughout, we denote by $M$ the matrix corresponding with $D$, as in the statement above.

We first prove that the stated condition is sufficient. 

Suppose $P_1\sim P_2$. Then we may assume that $r(x_1,y_1,z_1)=kr(x_2,y_2,z_2)$, for some $k\in\K$. But then 
$rD=0$ since we can multiply the first row of $M$ by $r$ and then this row becomes equal to $kr$ times the second row. Hence the three 
points $P_1,P_2,P_3$ are mutually non-neighboring. Suppose that $P_3$ neighbors the line $P_1P_2$. Then $P_3$ neighbors some point 
$P_3'=V^*(x_3',y_3',z_3')$ of $P_1P_2$, and we may assume that $r(x_3,y_3,z_3)=kr(x_3',y_3',z_3')$, $k\in\K$. By Lemma~\ref{N7}, 
there exist elements $a_1,a_2\in V$ such that $(x_3',y_3',z_3')=(a_1x_1+a_2x_2,a_1y_1+a_2y_2,a_1z_1+a_2z_2)$. At 
most one of $a_1,a_2$ is a scalar multiple of $s$ and hence we may assume that $a_1$ is not. Then we multiply the first row of $M$ by 
$a_1$ and the third by $r$ and the determinant does not vanish. But we may now add $a_2$ times the second row to the first, to obtain $kr$ 
times the last row. Hence the determinant does vanish, a contradiction.

Let $P_1P_2=V^*[a,b,c]$. Then, clearly, $ax_4+by_4+cz_4=a_3(ax_3+by_3+cz_3)$, which belongs to $V^*$ as $P_3$ does not neighbor $P_1P_2$. Likewise, $P_4$ does neither neighbor $L_2L_3$ nor $L_1L_3$.  

Now we show that the condition is necessary. So we suppose that $(P_1,P_2,P_3,P_4)$ is a proper ordered quadrangle.

Suppose $D\in \K r$. Developing the determinant with respect to the first row of $M$, we see by the dual of Lemma~\ref{N3} that the point $P_1$ neighbors the line $P_2P_3$, a contradiction.

Now the equation $(x~y~z)M=(x_4~y_4~z_4)$ has a solution (just multiply at the left with $M^{-1}$, which exists since the determinant is invertible), which we name $(a_1,a_2,a_3)$. Let $P_1P_2=V^*[a,b,c]$. Then, as before, $ax_4+by_4+cz_4=a_3(ax_3+by_3+cz_3)$, which belongs to $V^*$ as $P_3$ does not neighbor $P_1P_2$. This implies that $a_3\in V^*$. Likewise, $a_1,a_2\in V^*$ and the proof of the proposition is complete.  
\end{proof} 

This proposition shows that there are a lot of proper ordered quadrangles. We make this more explicit in the next subsection.
 
\subsection{Transitivity properties of $\cG(V)$}

Let $M$ be any $3\times3$ matrix with entries in $V$ such that the usual determinant is an invertible element in $V$. Then the 
transposed of the inverse exists, and we denote it by $M^*$. The set of all these matrices forms a group $\GL_3(V)$, and we will 
show some transitivity properties of $\GL_3(V)$ acting on $\cG(V)$.  The action is defined in a standard way: the 
point $V^*(x,y,z)$ is mapped onto the point $V^*(u,v,w)$, where $(u~v~w)=(x~y~z)M$. We denote $(u,v,w)=(x,y,z)^M$.

We start by showing that the above defined action is well defined.

\begin{prop}\label{welldefined}
The action of $\GL_3(V)$ on $\cG(V)$ as defined above is well-defined and maps lines to lines. Moreover, every element of $\GL_3(V)$ respects the neighboring relations. 
\end{prop}

\begin{proof} Let $M\in\GL_3(V)$. Let $V^*(x,y,z)$ be a point of $\cG(V)$, and put $(u,v,w)=(x,y,z)^M$. 
First we show that $V^*(u,v,w)$ is again a legal point of $\cG(V)$. Suppose not, then we may assume that $ru=rv=rw=0$. 
Hence $(rx~ry~rz)M=(0~0~0)$. Multiplying both sides at the right with the inverse of $M$, we see that $tx=ty=tz=0$, a 
contradiction. A standard argument now shows that the image of any point $V^*(x,y,z)$ under $M$ does not depend on the chosen 
representative $(x,y,z)$. 

Now we show that $M$ maps lines to lines. Indeed, let $L=V^*[a,b,c]$ be a line, then a point $V^*(x,y,z)$ is contained 
in $L$ if and only if $xa+yb+zc=0$. This happens if and only if $(x~y~z)MM^{-1}(a~b~c)^t=0$, and this is equivalent with saying 
that the point $V^*(x,y,z)^M$ is incident with the line $V^*[a',b',c']$, where $(a'~b'~c')=(a~b~c)M^*$.

Lemmas~\ref{N5} and~\ref{N6} imply that each element of $\GL_3(V)$ preserves the neighboring relations. 

The proof is complete.
\end{proof}

We continue by showing that $\GL_3(V)$ acts (sharply) transitively on proper ordered quadrangles. This is easy now:

\begin{theorem}
The group $\GL_3(V)$ acts sharply transitively on the family of all proper ordered quadrangles.
\end{theorem}

\begin{proof}
Let $(E_1,E_2,E_3,E)$ be the proper ordered quadrangle arising from the standard basis, i.e., $E_1=V^*(1,0,0), E_2=V^*(0,1,0), E_3=V^*(0,0,1)$ and $E=V^*(1,1,1)$. Let $(P_1,P_2,P_3,P_4)$ be an arbitrary proper ordered quadrangle, with $P_i:=V^*(x_i,y_i,z_i)$, $i=1,2,3,4$, and with $a_1,a_2,a_3\in V^*$ such that $(x_4,y_4,z_4)= \sum_{i=1}^3a_i(x_i,y_i,z_i)$, as guaranteed by 
Proposition~\ref{prop1}. Then the matrix 
$$M:=\left(\begin{array}{ccc} a_1x_1 & a_1y_1 & a_1z_1\\ a_2x_2 & a_2y_2 & a_2z_2 \\ a_3x_3 & a_3y_3 & a_3z_3\end{array}\right)$$ maps $(E_1,E_2,E_3,E)$ to $(P_1,P_2,P_3,P_4)$. Also, a standard argument shows that the stabilizer of $(E_1,E_2,E_3,E)$ in 
$\GL_3(V)$ is trivial. 
\end{proof}

Now, in order to derive from this theorem seemingly weaker transitivity properties, like flag-transitivity, one has for instance to show that every flag is contained in a proper ordered quadrangle. We do this in the next proposition.

\begin{prop}
Every point $P_1$ is contained in a flag $(P_1,L_1)$; every flag $(P_1,L_1)$ is contained in a triple $(P_1,L_1,P_2)$, where $P_2$ is a point on $L_1$ not neighboring $P_1$; every triple $(P_1,L_1,P_2)$, where $P_1,P_2$ are non-neighboring points on the line $L_1$, is contained in a quadruple $(P_1,L_1,P_2,L_2)$, where $L_2$ is a line through $P_2$ not neighboring $L_1$; every quadruple $(P_1,L_1,P_2,L_2)$, where $P_1,P_2$ are non-neighboring points on the line $L_1$, and $L_1,L_2$ are non-neighboring lines through the point $P_2$, is contained in a quintuple $(P_1,L_1,P_2,L_2,P_3)$, where $(P_1,P_2,P_3)$ is a proper ordered triangle with $P_3$ on $L_2$; every proper ordered triangle is contained in a proper ordered quadrangle.     
\end{prop}

\begin{proof}
The first assertion is trivial; the next three follow directly from Lemma~\ref{N4} (for the fourth, also use Lemma~\ref{N3});  the last assertion follows from Proposition~\ref{prop1}. 
\end{proof}

This now immediately implies the following transitivity properties.

\begin{cor}\label{flagtransitive}
The group $\GL_3(V)$ acts transitively on the family of all flags of $\cG(V)$, on the family of all triples  $(P_1,L_1,P_2)$, where $P_1,P_2$ are non-neighboring points on the line $L_1$, on the family of all quadruples $(P_1,L_1,P_2,L_2)$, where $P_1,P_2$ are non-neighboring points on the line $L_1$, and $L_1,L_2$ are non-neighboring lines through the point $P_2$, and on the family of proper ordered triangles. 
\end{cor}

\begin{rem} \rm If we define a \emph{proper anti-flag} as a pair $\{P,L\}$ consisting of a point $P$ not neighbouring a line $L$, 
then similarly as above, one easily shows that $\GL_3(V)$ acts transitively on the family of all proper anti-flags. But if we 
define an \emph{improper anti-flag} as a pair $\{P,L\}$ consisting of a point $P$ neighboring a line $L$, with $P$ not on $L$, then 
$\GL_3(V)$ does not necessarily acts transitively on the family of all improper anti-flags. Indeed, it does when $r=s$, but if 
$r\neq s$, then we need a semi-linear transformation with a companion automorphism of $V$ interchanging $\K r$ and $\K s$. 
Similarly, transitivity on the family of pairs of neighboring points is only guaranteed if $r=s$, and it does not occur if 
$r\neq s$, unless we extend $\GL_3(V)$ with the above mentioned semi-linear transformations. 

We do not need these remarks for the sequel, so we leave their (straightforward) proofs to the reader.
\end{rem}

\section{A $V$-set is a $\cC$-Veronesean set}

We now check that the $V$-sets defined by matrices verify the Mazzocca-Melone axioms.

To that end, we first show that the action of the group $\GL_3(V)$ on $\cG(V)$ extends to the $V$-set $X$ defined by matrices.

\begin{lemma}\label{groupaction}
Let $X$ be the $V$-set defined by matrices in $\PG(8,\K)$. Then there is a group $G\leq\GL_9(\K)$ acting on $X$, and an isomorphism $\theta: G\rightarrow\GL_3(V)$ such that for $g\in G$, 
and for $V^*(x,y,z)$ a point of $\cG(V)$, $$(V^*(x,y,z))^{\theta(g)}\mbox{ corresponds to }((x~y~z)^t(x~y~z)^\sigma)^g,$$
\end{lemma}

\begin{proof}
This follows from the fact that, if $M\in\GL_3(V)$ and $A$ is a Hermitian matrix over $V=\K\mathfrak{R}+\K\mathfrak{I}$ (the 
coefficient of $\mathfrak{R}$ will be referred to as the \emph{real} part, the one of $\mathfrak{I}$ as the 
\emph{imaginary} part) then the real and imaginary part of each entry of $M^tAM^\theta$ is a linear combination of the real and 
imaginary parts of the entries of $A$, with as coefficients quadratic expressions in the real and imaginary parts of the 
entries of $M$, also involving the elements $t,n\in\K$ for which $\mathfrak{I}^2-t\mathfrak{I}+n=0$. The latter uses the fact that 
$V$ is a quadratic algebra over $\K$.
\end{proof}

\begin{lemma}
The $V$-sets defined by matrices are $\cC$-Veronesean sets, when endowed with the set $\Xi$ of $3$-spaces obtained by considering the spans of the lines of $\cG(V)$, and where $\cC$ is the class of elliptic quadrics, of quadratic cones with removed vertex, or of hyperbolic quadrics according to whether $V_0$ is trivial, a $1$-dimensional subspace of $V$, or the union of two $1$-dimensional subspaces of $V$.
\end{lemma}
\begin{proof} Recall that the coordinates $(X_{0},\cdots,X_{8})$  of the point in $\PG(8,\K)$ corresponding with the point $V^*(M,N,L)$, with $M,N,L\in V$, where $V$ is viewed as matrix algebra, are $$(\det M,\det N,\det L,x(M^{\ad}N),y(M^{\ad}N), x(N^{\ad}L),y(N^{\ad}L), x(L^{\ad}M),y(L^{\ad}M)).$$

We first show the claim that the points on an arbitrary line of $\cG(V)$ span a $3$-space which intersects $X$ in a point set lying on a quadric (corresponding to the class $\cC$). Indeed, by Corollary~\ref{flagtransitive}, we may consider the line $V^*[0,0,1]=V^*[0,0,\mathfrak{R}]$. An arbitrary point $V^*(M,N,L)$ on that line has $L=0$, and this translates easily into $X_{2}=X_{5}=X_{6}=X_{7}=X_{8}=0$, which is our wanted $3$-space $\xi$. Now suppose we have a point $V^*(M,N,L)$ of $X$ in this $3$-space. Then the equations tell us that $L^{\ad}M=N^{\ad}L=\det L=0$. These equations imply $L^{\ad}(M,N,L)=(0,0,0)$, hence $L=0$. So we already have that $\xi\cap X$ corresponds to the points of $L$. Using $\det M\det N=\det(M^{\ad}N)$, the matrix representation of an arbitrary point of $L$ yields the quadratic equation $X_{0}X_{1}=X_{3}^{2}+tX_{3}X_{4}+nX_{4}^{2}$, which is the equation of our desired quadric $H$.  Our claim follows.

Now we show that every regular point $P$ of $H$ is a point of the line $V^*[0,0,\mathfrak{R}]$, and no singular point is.

Suppose $P$ has coordinates $(x_0,x_1,0,x_3,x_4,0,0,0,0)$. First assume that $x_0\neq 0$; then we may assume $x_0=1$. Then one easily calculates that $P$ corresponds to the point $V^*(\mathfrak{R},x_3\mathfrak{R}+x_4\mathfrak{I},0)$. Similarly if $x_1\neq 0$, then $P$ corresponds to a point of $V^*[0,0,\mathfrak{R}]$. If $H$ is elliptic, then there are no more points; if $H$ is a cone, then the only point of $H$ with $x_0=x_1=0$ is the vertex and we claim this does not belong to $V^*[0,0,\mathfrak{R}]$. Indeed, by Corollary~\ref{flagtransitive} it would follow that every point of $H$ is singular, a contradiction. Finally, if $H$ is hyperbolic, then there are precisely two points with $x_0=x_1=0$, and if one or both of them would not belong to $V^*[0,0,\mathfrak{R}]$, then the collineation group stabilizing the set of points $H$ belonging to $V^*[0,0,\mathfrak{R}]$ could not act transitively, contradicting Corollary~\ref{flagtransitive}.   

Hence we have shown that the intersection of $X$ with the space generated by the points on $V^*[0,0,\mathfrak{R}]$ is a member of $\cC$. Now Property (V1) is obvious. 

To prove property (V3*), we start by using the above equation of $H$ in $\xi$. The tangent space at $P=(1,0,0,0,0,0,0,0,0)$ to $H$ is the plane given by the equations
$X_{1}=X_{2}=X_{5}=X_{6}=X_{7}=X_{8}=0$. 

Similarly the tangent space at $P$ to the quadric formed by the points on the line $V^*[0,\mathfrak{R},0]$  is given by the equations $X_{1}=X_{2}=X_{3}=X_{4}=X_{5}=X_{6}=0$.
The space spanned by these two tangent spaces is the 4-dimensional space $\Pi_{4}$ determined by the equations $X_{1}=X_{2}=X_{5}=X_{6}=0$.
We must prove that every other tangent space at $p$ to a line is contained in $\Pi_{4}$. 

Consider an arbitrary line through $p$; this must intersect the line $V^*[\mathfrak{R},0,0]$ in some point $V^*(0,N,L)$, $N,L\in V$. We may assume $b\neq 0$ and thus $N=sL$. This implies that $X_{1}=s^{2}X_{2}$. Moreover it follows that $X_{6}=0$.
Hence we obtain the equation $s^{2}X_{1}^{2}=X_{5}^{2}$. But for the tangent space to $p$ to this quadric we clearly have again $X_{1}=0$, and hence this space is contained in $\Pi_{4}$.

Finally, we show property (V2) and first the non-neighboring situation.
By Corollary~\ref{flagtransitive} we may take the lines $V^*[\mathfrak{R},0,0]$ and $V^*[0,0,\mathfrak{R}]$. Similarly as above, the corresponding subspaces in $\PG(8,\K)$ have equations $X_0=X_3=X_4=X_7=X_8=0$ and $X_{2}=X_{5}=X_{6}=X_{7}=X_{8}=0$, respectively. Clearly, these meet only in the point with coordinates  $(0,1,0,0,0,0,0,0,0)$, which corresponds to the point $V^*(0,\mathfrak{R},0)$ of $ \cG(V)$. 

Now we consider the neighboring situation. By Corollary~\ref{flagtransitive} and Definition~ref{defnei}, we may consider the lines $V^*[0,0,\mathfrak{R}]$ and $V^*[0,kM,\mathfrak{R}+\ell M]$, where $M\in V_0\setminus\{0\}$. Put $M=a\mathfrak{R}+b\mathfrak{I}$, with $a,b\in\K$. One calculates that the points $V^*(0,\mathfrak{R}+\ell M,-k M)$ and $(1, \mathfrak{R}+\ell M,-k M)$ belong to $V^*[0,kM,\mathfrak{R}+\ell M]$, and the last four coordinates of these points in $\PG(8,\K)$ are $(-ka,-kb,0,0)$ and $(-ka,-kb,kx+kyt,-ky)$, respectively. This implies, in view of the fact that the last four coordinates of every point of $V^*[0,0,\mathfrak{R}]$ are $0$, that the set of points of $\PG(8,\K)$ corresponding with the set of points of $\cG(V)$ on the union of the lines  $V^*[0,0,\mathfrak{R}]$ and $V^*[0,kM,\mathfrak{R}+\ell M]$ span a subspace of dimension at least $5$. But since these lines meet in at least $2$ points, this dimension is exactly $5$ and the corresponding $3$-spaces meet in a line $T$. Now, the intersection points of these two lines can be generically written as $V^*(m\mathfrak{R},M',0)$, with $m\in\K$ and where $MM'=0$, $M'\neq 0$ (we overlook at most one point, and that happens precisely when $M$ and $M$ are linearly independent; then we have the additional point $V^*(M,M',0)$). Writing the coordinates of such generic point in $\PG(8,\K)$, we see that $m$ appears linearly, so that we miss at most one point of $T$. This shows Property (V2).
\end{proof}




To finish the proof of the Main Result, it now suffices to show the following uniqueness result.

\begin{prop}
Every (abstract) ovoid, tube and hypo contained in a Hermitian, 
Hjelmslevian and Segrean Veronesean set, respectively, is contained in an elliptic, cylindric and hyperbolic space, respectively.
\end{prop}

\begin{proof}
If the groundfield has order at least $3$, then there is a uniform approach. Let $X$ be a Hermitian, Hjelmslevian or Segrean 
Veronesean set with $\Xi$ the set of corresponding elliptic, cylindric and hyperbolic spaces, respectively. Let $Q$ be an ovoid, tube or hypo contained in $X$, and not contained in a member 
of $\Xi$. Let $x,y\in X$ be two points such that the projective line spanned by $x$ and $y$ is not entirely contained in $Q$ (and then $x$ and $y$ are the only two points of $X$ on that line). Such 
points exist since otherwise one easily sees that the $3$-space generated by $Q$ is entirely contained in $X$, a contradiction. 
Let $H$ be the unique elliptic quadric, tube or hypo, respectively, contained in a member $\xi$ of $\Xi$ and 
containing both $x$ and $y$. Choose an oval $O$ in $Q$ through $x$ and $y$. If $O$ does not belong to $\xi$, then we can select two 
points $x',y'$ in $O$ distinct from $x,y$, and a member $\xi'$ of $\Xi$ containing $x',y'$. Since $\xi'\neq\xi$, the axioms imply 
$\xi'\cap\xi\subseteq X$, a contradiction (for the Hjelmslevian case, it really follows from Lemma~\ref{lemmaH1}). 

Hence $O\subseteq \xi$ and so $O\subseteq H$. But we do the same for a second oval through $x,y$ and obtain $Q\subseteq \xi$. 

Now we consider the cases of Hermitian and Segrean Veronesean sets over the field $\mathbb{F}_2$. In the Segrean case, 
it is easily seen that a hypothetical hypo $Q$ not contained in a member of $\Xi$ shares two generators with a hypo $H$ contained 
in a member of $\Xi$. Now taking two appropriate points of $Q$ outside $H$ and considering a hyperbolic space through them leads 
to a contradiction, as above. In the Hermitian case, the same contradiction arises when a hypothetical ovoid $Q$ in $X$ not 
contained in a member of $\Xi$ intersects an elliptic space in an ovoid; hence we may assume that it intersects some elliptic 
space $\xi$ in just two points $x,y$. But the projection from $\xi$ of $X\setminus\xi$ is injective and contained in an affine space 
of dimension $4$, see \cite{Sch-Mal:11}, and the three points of $Q$ distinct from $x,y$ would project on a line; a contradiction.

For the Hjelmslevian Veronesean set $X$ over $\mathbb{F}_2$, we note that it is easily seen that $4$ points in general position can never lie in a common plane. Suppose we have $4$ points of $X$ in a common plane, then three of them must lie on singular lines through a common vertex $t$. One now easily checks that the space generated by all points on a singular line through $t$ (a 12-set in the terminology of the last part of Section~\ref{sec:HVS}) does not contain more points than those 12 points; hence the fourth point must also be in that set. It now follows easily that a given tube must be contained in a cylindric space.

The proposition, and hence our Main Result, is completely proved.
\end{proof}


\section{Further motivation}

In this section we put our results in a broader perspective, at the same time motivating why we look at these generalizations of Hermitian Veronese varieties, and explaining why we stated the Axioms~(V1), (V2), (V3) and (V3*) in such generality.

Consider an alternative composition division algebra $\mathbb{A}$, and suppose $\mathbb{A}$ is either $4$- or $8$-dimensional over 
its center (hence we have a quaternion or octonion algebra). Then there is a unique projective plane $\mathbb{P}_2(\mathbb{A})$ 
associated with $\mathbb{A}$, and this projective plane corresponds to a real form of an algebraic group of absolute type 
either $\mathsf{A}_5$ or $\mathsf{E}_6$. The (standard) Veronese varieties associated to these planes arise from a standard 
representation of the spherical buildings of type $\mathsf{A}_5$ and $\mathsf{E}_6$, respectively, in a 15- and 27-dimensional 
vector space, respectively, by suitably restricting coordinates. These Veronese varieties can also be constructed directly as certain 
Hermitian matrices with entries in $\mathbb{A}$. Now, if we take for $\mathbb{A}$ the split version of the algebra, then we obtain 
in this way  the aforementioned representations of the respective spherical buildings.

So, we may conclude that there is a uniform construction of all Veronese varieties corresponding to composition division 
algebras, and certain representations of spherical buildings. Our ultimate aim is to characterize all these structures in a uniform 
way by the axioms (V1), (V2) and (V3*). Moreover, by allowing not only split and non-split algebras, but also certain other 
quadratic variants $V$, for which the associated geometry $\cG(V)$ as defined in Section 1 has the properties that any two points lie 
on at least one common line and any pair of lines has at least one point in common, we hope to include some less standard varieties, 
such as the Hjelmslev planes of the present paper, which are still connected to Tits-buildings. Indeed, the Hjelmslev planes of level 
2 are precisely the structures emerging from the set of vertices at distance 2 from a given vertex in an affine building of type 
$\widetilde{\mathsf{A}}_2$, see \cite{Han-Mal:89,Mal:87}. 

The geometries referred to in the previous paragraph correspond to the second row of the Freudenthal-Tits magic square (over any field). In 
particular, our characterization should ultimately include a characterization of the standard module for groups of type 
$\mathsf{E}_6$ over any field.  It would be rather remarkable that this intricate geometric structure allows to be caught by a rather 
simple and short list of axioms.

By intersecting the models of the $\mathcal{C}$-Veronesean sets corresponding to the second row of the Freudenthal-Tits magic 
square with an appropriate hyperplane (codimension 1 space), we obtain models for the geometries and groups in the first row of 
this square. A second aim of ours is to find adequate axioms for these representations.
These axioms should have the same spirit as those considered in the present paper. This would yield common characterizations of 
certain representations of buildings of types $\mathsf{A}_2$ (the analogues of those considered in \cite{Tha-Mal:00} for arbitrary 
fields), $\mathsf{C}_3$ (the line Grassmannian representation) and $\mathsf{F}_4$ (the ordinary 25-dimensional module). 

Of course, further perspectives include the third and the final row of the Freudenthal-Tits magic square, where, amongst others, non-embeddable polar spaces and buildings of types $\mathsf{E}_7$ and $\mathsf{E}_8$ appear. The former, although already apparent in Tits' original approach, were formally proved to be in the Magic Square by M\"uhlherr in \cite{Mue:90}. This indicates that a geometric approach, as started here is worthwhile to pursue.  

Finally we note that \cite{Chaput,Lan:03} contain related approaches to the magic square, describing the algebraic geometry and representation theory associated to the Freudenthal-Tits magic square (hence restricted to certain fields, for instance not over finite fields).


\textbf{Address of the authors}\\
Jeroen Schillewaert\\
D\'epartement de Math\'ematique,
Universit\'e Libre de Bruxelles\\
U.L.B., CP 216,
Bd du Triomphe,
B-1050 Bruxelles,
BELGIQUE\\
\texttt{jschille@ulb.ac.be}\\ {}\\
Hendrik Van Maldeghem\\
Department of Mathematics,
Ghent University,\\
Krijgslaan 281, S22,
B-9000 Ghent,
BELGIUM\\
\texttt{hvm@cage.ugent.be}

\end{document}